\documentclass{elsarticle}
\usepackage{amssymb,latexsym,theorem}
\usepackage{amsmath,amsfonts}
\usepackage{enumerate}
\usepackage{array}
\usepackage[british]{babel}
\newcommand{\G}{\mathcal{G}}

\newcommand{\cS}{\mathcal{S}}

\newcommand{\cB}{\mathcal B}
\newcommand{\cU}{\mathcal U}
\newcommand{\cL}{\mathcal{L}}
\newcommand{\cM}{\mathcal{M}}
\newcommand{\cC}{\mathcal{C}}
\newcommand{\cV}{\mathcal{V}}
\newcommand{\cT}{\mathcal{T}}
\newcommand{\cX}{\mathfrak{X}}
\newcommand{\fB}{\mathfrak{B}}
\newcommand{\cN}{\mathcal{N}}
\newcommand{\fX}{\mathfrak{X}}
\newcommand{\cG}{\mathcal{G}}
\newcommand{\ccQ}{\mathcal{Q}}

\newcommand{\trace}{\mbox{\itshape trace}}
\newcommand{\diag}{\mbox{\itshape diag}}
\newcommand{\spin}{\text{\itshape spin}}

\newcommand{\baM}{\overline{M}}
\newcommand{\PG}{\mathrm{PG}}

\newcommand{\F}{\mathbb{F}}
\newcommand{\FF}{\mathbb{F}}

\newcommand{\N}{\mathcal{N}}
\newcommand{\cP}{\mathcal{P}}

\newtheorem{lemma}{Lemma}[section]

\newtheorem{theo}[lemma]{Theorem}
\newtheorem{co}[lemma]{Corollary}
\newtheorem{mth}{Main Result}
\newtheorem{theorem}{Theorem}

\newproof{proof}{Proof}

\begin{document}
\begin{frontmatter}
\title{Codes and caps from orthogonal Grassmannians}
\author[IC]{Ilaria Cardinali}
\ead{ilaria.cardinali@unisi.it}
\address[IC]{Department of Information Engineering and Mathematics, University of Siena,
Via Roma 56, I-53100, Siena, Italy}
\author[LG]{Luca Giuzzi\corref{cor1}}
\ead{giuzzi@ing.unibs.it}
\address[LG]{DICATAM - Section of Mathematics,
University of Brescia,
Via Valotti 9, I-25133, Brescia, Italy}
\cortext[cor1]{Corresponding author. Tel. +39 030 3715739; Fax. +39 030 3615745}
\begin{abstract}
In this paper we investigate linear
error correcting codes and projective caps related
to the Grassmann embedding $\varepsilon_k^{gr}$
of an orthogonal Grassmannian $\Delta_k$.
In particular,
we determine some of the parameters of the codes arising
from the projective system determined by
$\varepsilon_k^{gr}(\Delta_k)$.
We also study
special sets of points of $\Delta_k$ which are met
by any line of $\Delta_k$ in at most $2$ points and we show
that their image under the Grassmann embedding $\varepsilon_k^{gr}$
is a projective cap.
\end{abstract}
\begin{keyword}
  Polar Grassmannians \sep dual polar space \sep
  embedding \sep
  error correcting code \sep cap \sep Hadamard matrix
  \sep Sylvester construction.

\MSC[2010] 51A50 \sep 51E22 \sep 51A45
\end{keyword}
\end{frontmatter}
\section{Introduction}\label{Introduction}
The overarching theme of this paper is the behaviour of the image of the
Grassmann embedding $\varepsilon_k^{gr}$ of an orthogonal Grassmannian
$\Delta_k$ with $k\leq n$ with respect to linear subspaces of either
maximal or minimal dimension.
In the former case,
we obtain the parameters of the linear error correcting
codes arising from the projective system determined by the pointset
$\varepsilon_k^{gr}(\Delta_k)$ and
provide a bound on their minimum distance.
%
In the latter,
we consider and construct special sets of points of $\Delta_k$
 that are met by each line of $\Delta_k$ in at
most $2$ points and show that the Grassmann embedding
maps these sets in projective caps. Actually, an explicit construction of
a family of such sets, met by any line in at most $1$ point,
is also provided, and a link with Hadamard
matrices is presented.

The introduction is organised as follows:
in Subsection~\ref{Sect-1.1}
we shall provide a background on embeddings of orthogonal
Grassmannians; Subsection~\ref{Sect-1.2} is devoted to codes
arising from projective systems, while in Subsection~\ref{Sect-1.3}
we summarise our main results and outline the structure of the paper.

\subsection{Orthogonal Grassmannians and their embeddings}
\label{Sect-1.1}

Let $V := V(2n+1,q)$ be a $(2n+1)$--dimensional vector space over a finite
field $\FF_q$ endowed with a non--singular quadratic form  $\eta$ of
Witt index $n$.
For $1\leq k\leq n$, denote by $\G_k$ the $k$--Grassmannian
of $\PG(V)$ and by $\Delta_k$ the $k$--polar Grassmannian
associated to $\eta$,  in short the latter will be
called an \emph{orthogonal Grassmannian}.
We recall that $\G_k$ is the point--line geometry whose
points are the $k$--dimensional subspaces of $V$ and whose lines are
sets of the form
\[\ell_{X,Y} := \{Z \mid X \subset Z\subset Y, ~\dim(Z) = k\},\]
where $X$ and $Y$ are  any two subspaces of $V$ with $\dim(X) = k-1$,
$\dim(Y) = k+1$ and $X\subset Y$.

The orthogonal Grassmannian $\Delta_k$ is the
proper subgeometry of $\G_k$ whose
points are the $k$--subspaces of $V$ totally
singular for $\eta$.
 For $k<n$ the lines of $\Delta_k$ are exactly the
lines $\ell_{X,Y}$ of $\G_k$ with $Y$ totally singular; on the other hand,
when $k=n$ the
lines of $\Delta_n$ turn out to be the sets
\[\ell_{X}:=\{Z\mid X \subset Z\subset X^{\perp} , ~\dim(Z) = n,
~ Z ~\text{totally ~singular}\}\]
with $X$  a totally singular
$(n-1)$--subspace of $V$ and $X^\perp$ its orthogonal  with
respect to $\eta$. Note that the points of $\ell_{X}$ form a conic in the
projective plane $\PG(X^\perp/X)$.
Clearly, $\Delta_1$ is just the
orthogonal polar space of rank $n$ associated to $\eta$;
the geometry $\Delta_n$ can be
regarded as its dual and is thus
called the \emph{orthogonal dual polar space} of rank $n$.
Recall that the size of the point--set of $\Delta_k$ is
$\prod_{i=0}^{k-1}\frac{q^{2(n-i)}-1}{q^{i+1}-1}$;
see~e.g.~\cite[Theorem 22.5.1]{HT91}.

Given a point--line geometry $\Gamma=(\cP,\mathcal{L})$ we say that an
injective map $e\colon\cP\to\PG(V)$ is a
\emph{projective embedding} of $\Gamma$ if the following conditions hold:
\begin{enumerate}[(1)]
\item $\langle e(\cP)\rangle =\PG(V)$;
\item \label{c2a} $e$ maps any line of $\Gamma$ onto a projective line.
\end{enumerate}
Following~\cite{VM}, see also~\cite{CP1},
when condition (\ref{c2a}) is replaced by
\begin{itemize}
\item[(\ref{c2a}')] $e$ maps any line of $\Gamma$ onto a
 non--singular conic of $\PG(V)$ and for all $l\in\cL$,  $\langle e(l)\rangle \cap e(\cP)=e(l)$
\end{itemize}
we say that $e$ is a
\emph{Veronese embedding} of $\Gamma$.

The dimension $\dim(e)$ of an
embedding $e\colon \Gamma \rightarrow \PG(V)$,
either projective or Veronese,
is the dimension of the vector space $V$.
When $\Sigma$ is a proper subspace of $\PG(V)$
such that $e(\Gamma)\cap \Sigma=\emptyset$ and $\langle e(p_1),
e(p_2)\rangle\cap \Sigma =\emptyset$ for any two distinct points $p_1$
and $p_2$ of $\Gamma$, then it is possible to define a new embedding
$e/\Sigma$ of $\Gamma$ in the quotient space $\PG(V/\Sigma)$ called
the \emph{quotient of $e$ over $\Sigma$} as $(e/\Sigma)(x)=\langle
e(x), \Sigma\rangle/\Sigma$.

Let now $W_k :=\bigwedge^k V$. The \emph{Grassmann} 
 or \emph{Pl\"ucker}
 \emph{embedding}
$e_k^{gr}:\G_k\rightarrow \PG(W_k)$ 
maps the arbitrary $k$--subspace
$\langle v_1,v_2,\ldots, v_k\rangle$ of $V$ (hence a point of $\G_k$)
to the point $\langle
v_1\wedge v_2\wedge\cdots \wedge v_k \rangle$ of $\PG(W_k)$. Let
$\varepsilon_k^{gr}:= {e_k^{gr}}|_{\Delta_k}$ be the restriction of
$e_k^{gr}$ to $\Delta_k$. For $k<n$, the mapping
$\varepsilon_k^{gr}$ is a projective embedding of $\Delta_k$ in the
subspace $\PG(W_k^{gr}):=\langle\varepsilon_k^{gr}(\Delta_k)\rangle$
of $\PG(W_k)$ spanned by $\varepsilon_k^{gr}(\Delta_k)$. We call
$\varepsilon_k^{gr}$ the \emph{Grassmann} embedding of $\Delta_k$.

If $k=n$, then $\varepsilon_n^{gr}$ is a Veronese embedding
and  maps the lines of $\Delta_n$ onto non--singular
conics of $\PG(W_n)$.
The dual polar space $\Delta_n$ affords
also a projective embedding of dimension $2^n,$ namely
the spin embedding $\varepsilon_n^{\spin}$;
for more details we refer the reader to either \cite{Chevalley} or ~\cite{BC}.

Let now $\nu_{2^n}$ be the usual quadric Veronese map
$\nu_{2^n}\colon V(2^n,{\mathbb{F}})\to V({{2^n+1}\choose 2},{\mathbb{F}})$
given by
\[(x_1,\ldots,x_{2^n})\to
(x_1^2,\ldots,x_{2^n}^2,x_1x_2,\ldots, x_1x_{2^n},x_2x_3,\ldots,
x_2x_{2^n}, \ldots, x_{2^n-1}x_{2^n}).\]
It is well known that $\nu_{2^n}$
defines a Veronese embedding of the point--line geometry
$\PG(2^n-1,\FF)$ in $\PG({{2^n+1}\choose 2}-1,\FF)$, which will
be also denoted by $\nu_{2^n}$.

The composition $\varepsilon^{{vs}}_n
:= \nu_{2^n}\cdot\varepsilon^{\spin}_n$ is a Veronese
embedding of $\Delta_n$ in a subspace $\PG(W_n^{vs})$ of
$\PG({{2^n+1}\choose 2}-1,\FF)$: it is called the \emph{Veronese--spin
embedding} of $\Delta_n$.

We recall some results from~\cite{CP1} and~\cite{CP2} on the Grassmann
and Veronese--spin embeddings of $\Delta_k$, $k\leq n$. Observe that
these results hold over arbitrary fields, even if in the present
paper we shall be concerned just with the finite case.

\begin{theorem}\label{emb1} Let $\FF_q$ be a finite field with  $char(\FF_q)\neq2$. Then
\begin{enumerate}[(1)]
\item\label{emb1:pt1} $\dim ({\varepsilon}_k^{gr})={{2n+1}\choose k}$ for any
$n\geq 2$, $k\in\{1,\ldots,n\}$.
\item\label{emb1:pt2}  $\varepsilon_n^{{vs}}\cong {\varepsilon}_n^{gr}$ for
any $n\geq 2$.
\end{enumerate}
\end{theorem}
When $char(\FF_q)=2$ there exist two subspaces $\N_1\supset \N_2$ of $\PG(W_n^{vs})$, called
\emph{nucleus subspaces}, such that the following holds.
 \begin{theorem}\label{emb2} Let $\FF_q$ be a finite field with  $char(\FF_q)=2$. Then,
\begin{enumerate}[(1)]
\item\label{emb2:pt1} $\dim ({\varepsilon}_k^{gr})={{2n+1}\choose
k}-{{2n+1}\choose {k-2}}$ for any $k\in\{1,\ldots,n\}$.
\item\label{emb2:pt2} $\varepsilon_n^{{vs}}/\N_1\cong
{\varepsilon}_n^{\spin}$ for any $n\geq 2$.
\item\label{emb2:pt3} $\varepsilon_n^{{vs}}/\N_2\cong
{\varepsilon}_n^{gr}$ for any $n\geq 2$.
\end{enumerate}
\end{theorem}

\subsection{Projective systems and Codes}
\label{Sect-1.2}
Error correcting codes are an essential component to any efficient
communication system, as they can be used in order to guarantee
arbitrarily low probability of mistake in the reception of messages
without requiring noise--free operation; see~\cite{MS}.
An $[N,K,d]_q$ projective system $\Omega$ is a set of $N$ points
in $\PG(K-1,q)$ such that for any hyperplane $\Sigma$ of $\PG(K-1,q)$,
\[ |\Omega\setminus\Sigma|\geq d. \]
Existence of $[N,K,d]_q$ projective systems is equivalent to that of
projective linear codes with the same parameters; see~\cite{H1,BvE,DS,TVN}.
Indeed, given a projective system $\Omega=\{P_1,\ldots,P_N\}$, fix
a reference system $\fB$ in
$\PG(K-1,q)$ and consider the matrix $G$ whose columns are
the coordinates of the points of $\Omega$ with respect to $\fB$.
Then, $G$ is the generator matrix of an $[N,K,d]$ code over
$\FF_q$, say $\cC=\cC(\Omega)$,  uniquely defined up to code
equivalence.
Furthermore, as any word $c$ of $\cC(\Omega)$ is of the form
$c=mG$ for some row vector $m\in\FF_q^K$,
it is straightforward to see that the number of
zeroes in $c$ is the same as the number of points of $\Omega$
lying on the hyperplane of equation
$m\cdot x=0$ where $m\cdot x=\sum_{i=1}^K m_ix_i$ and
$m=(m_i)_1^K$, $x=(x_i)_1^K$.
In particular, the minimum distance $d$ of $\cC$ is
\begin{equation}
\label{eq-codes}
  d=\min_{\begin{subarray}{c}\Sigma\leq\PG(K-1,q)\\ \dim\Sigma=K-2 \end{subarray}}
\left(|\Omega|-|\Omega\cap\Sigma|\right).
\end{equation}
\par
The link between incidence structures $\cS=(\cP,\cL)$ and codes
is deep and
it dates at least to~\cite{PLJ}; we refer the interested reader
to~\cite{AK,CvL} and~\cite{TLA05} for more details. Traditionally,
two basic approaches have been proven to be
most fruitful: either to regard the incidence matrix of $\cS$ as
a generator matrix for a binary code, see for instance~\cite{V07,GLS},
or to consider an embedding of
$\cS$ in a projective space and study either the code
arising from the projective system thus determined or its dual;
see~e.g.~\cite{B,CD,GP12} for codes related to the Segre embedding.

Codes based on projective Grassmannians belong to this latter class.
They have been first introduced
in~\cite{R1} as generalisations of Reed--Muller codes of the first order
and whenceforth extensively investigated; see
also~\cite{R2,N96,GL2001,GPP2009}.

\subsection{Organisation of the paper and main results}
\label{Sect-1.3}
In  Section~\ref{lcD} we study linear
codes associated with the projective system
 $\varepsilon_k^{gr}(\Delta_k)$ determined
by the embedding $\varepsilon_k^{gr}$.

We recall that a  \emph{partial spread} of a non-degenerate quadric
is a set of pairwise disjoint generators;
see also \cite[Chapter 2]{DbS}.
A partial spread $\mathcal{S}$ is a \emph{spread} if all the points
of the quadric are covered by exactly one of its elements. We recall that
for $q$ odd the quadrics $Q(4n,q)$ do not admit spreads.

\begin{mth}\label{main1}
  Let $\cC_{k,n}$ be the code arising from the projective
  system $\varepsilon_k^{gr}(\Delta_k)$ for $1\leq k< n$.
  Then, the parameters of $\cC_{k,n}$ are
      \[N=
      \prod_{i=0}^{k-1}\frac{q^{2(n-i)}-1}{q^{i+1}-1},
      \qquad K=\left\{\begin{array}{ll}
          \binom{2n+1}{k} & \mbox{for $q$ odd} \\
          \binom{2n+1}{k}-\binom{2n+1}{k-2} &
          \mbox{for $q$ even,}\\
          \end{array}\right. \,\,\,\, \]
        \[ d\geq \psi_{n-k}(q)(q^{k(n-k)}-1)+1, \]
where $\psi_{n-k}(q)$ is the maximum size of a (partial) spread of the
parabolic quadric $Q(2(n-k),q)$.
\end{mth}
We observe that, in practice, we expect the bound on the minimum
distance not to be sharp.
 As for codes arising from dual polar spaces of small rank,
 we have the following result where the minimum distance is exactly
 determined.
 \begin{mth}\label{main2}
\begin{enumerate}[(i)]
    \item\label{mt2:i}
      The code $\cC_{2,2}$ arising from a dual polar space of rank $2$ has parameters
      \[N=(q^2+1)(q+1),\qquad
      K=\left\{\begin{array}{ll}
          10 & \mbox{for $q$ odd} \\
          9  & \mbox{for $q$ even,}
          \end{array}\right. \qquad d=q^2(q-1).
       \]
    \item\label{mt2:ii}
      The code $\cC_{3,3}$ arising from a dual polar space of rank
      $3$ has parameters
      \begin{small}
      \[
      \begin{array}{ll l}
        N=(q^3+1)(q^2+1)(q+1),\,\,\,& K=35,\,\,\, & d=q^2(q-1)(q^3-1)\,\,\,  \mbox{ for $q$ odd} \\
         & and& \\
        N=(q^3+1)(q^2+1)(q+1),\,\,\,& K=28,\,\,\, & d=q^5(q-1)\,\,\, \mbox{ for $q$ even}.\\
       \end{array}\]
       \end{small}
    \end{enumerate}
\end{mth}

In Section~\ref{sec:2},
we introduce the notion of $(m,v)$--set of a
partial linear space and the notion of polar $m$--cap of $\Delta_k$.
We prove that the image of a polar $m$--cap under the
Grassmann embedding is a projective cap; see also~\cite{EMS} for caps contained in Grassmannians.
\begin{mth}\label{main3}
  Suppose $1\leq k\leq n$.
  Then,
  \begin{enumerate}
    \item
      the image $\varepsilon_k^{gr}(\cX)$ of any polar $m$--cap
      $\cX$ of $\Delta_k$
      is a projective cap
      of $\PG(W_k)$;
    \item the image of
      $\varepsilon_n^{gr}(\Delta_n)$ is
      a projective cap.
\end{enumerate}
\end{mth}
In Section~\ref{construction}
we give an explicit construction of some
$(2^r,1)$--sets contained in $\Delta_k$ with
$r\leq\lfloor k/2\rfloor$. This leads to the following
theorem.
\begin{mth}\label{main4}
\label{m:cc}
For any $r\leq\lfloor k/2\rfloor$, the polar
Grassmannian $\Delta_k$ contains a polar $2^r$-cap $\cX$.
\end{mth}
Finally, in Section~\ref{Hadamard-Sylvester},
we consider matrices $H$ associated to the polar caps $\fX$ of
Main Result~\ref{m:cc} and prove that they are
of Hadamard type. It is well known that these matrices lead
to important codes; see \cite[Chapter 3]{KJH}.
Then, it is shown that it is possible
to introduce an order on the points of $\fX$ as to
guarantee the matrix $H$ to be in Sylvester form, thus obtaining
a first-order Reed--Muller code; see \cite[page 42]{KJH}.

\section{Linear Codes associated to $\Delta_k$}
\label{lcD}
\subsection{General case}
By Theorem~\ref{emb1}, for $q$ odd and $1\leq k\leq n,$
the Grassmann embedding $\varepsilon^{gr}_k$ of $\Delta_k$ into
$\PG(\bigwedge^{k}V)$ has dimension ${2n+1}\choose k$; by Theorem~\ref{emb2},
for $q$ even and $1\leq k\leq n$,
$\dim (\varepsilon^{gr}_k)={{2n+1}\choose{k}}-{{2n+1}\choose{k-2}}$.
As such, the image of $\varepsilon^{gr}_k$
determines
a projective code $\cC^{gr}_{k,n}=\cC(\varepsilon_{k}^{gr}(\Delta_k))$.
Observe that $\cC_{k,n}^{gr}$ can be obtained by the full Grassmann code,
see~\cite{N96}, by deleting a suitable number of components; however,
this does not lead to useful bounds on the minimum distance.
The following lemma is a direct consequence of the definition of
$\cC^{gr}_{k,n}$.
\begin{lemma}\label{lemma-parametri}
The code $\cC^{gr}_{k,n}$ is a $[N,K]$-linear code with
\[N=
\prod_{i=0}^{k-1}\frac{q^{2(n-i)}-1}{q^{i+1}-1},\qquad
K=\begin{cases}
\binom{2n+1}{k} & \mbox{for $q$ odd} \\
\binom{2n+1}{k}-\binom{2n+1}{k-2} & \mbox{for $q$ even}.
\end{cases} \]
\end{lemma}
Given any $m$--dimensional subspace $X\leq V$ with $m>k$,
in an analogous way as the one followed to
define the $k$--Grassmannian $\G_k$ of $\PG(V)$ in Section~\ref{Introduction},
we introduce the
$k$--Grassmannian $\G_k(X)$ of $\PG(X)$.
More in detail,  $\G_k(X)$ is the point--line geometry having as points
the $k$--dimensional subspaces of $X$ and as lines exactly the lines of
$\mathcal{G}_k$ contained in $\G_k(X)$.

 The following lemma is straightforward.
\begin{lemma}
\label{l:mnd}
Suppose $X$ to be a totally singular subspace with $\dim X=m$ and
$k<m<n$; write
$W_k(X)=\langle\varepsilon^{gr}_k(\G_k(X))\rangle\leq W_k$.
Then,
 \[ \varepsilon^{gr}_k(\G_k(X))=\varepsilon^{gr}_k(\Delta_k)\cap W_k(X)
 =e_k^{gr}(\G_k(X)). \]
\end{lemma}
 
Let $X$ be a $k$--dimensional subspace of $V$ contained
in the non-degenerate parabolic quadric  $Q(2n,q)\cong\Delta_1.$
Define the \emph{star $St(X)$ of $X$}
as the set formed by the $i$--dimensional subspaces
of $Q(2n,q)$, $k<i\leq n$, containing $X$. It is well--known that
$St(X)$ is isomorphic to
a parabolic  quadric $Q(2(n-k),q)$; see \cite[Chapter 7]{T}.

Denote by $\psi_r(q)$ the maximum size of a (partial) spread of
$Q(2r,q)$.  Recall that for $q$ even, $Q(2r,q)$
admits a spread; thus $\psi_r(q)=q^{r+1}+1$. For $q$ odd a
general lower bound is $\psi_r(q)\geq q+1$, even if improvements
are possible in several cases; see \cite{DKMS}, \cite[Chapter 2]{DbS}.
\begin{theo}
\label{t:mnd}
If $k<n$,  the minimum distance $d$ of $\cC_{k,n}^{gr}$ is at least
  $$s=\psi_{n-k}(q)(q^{k(n-k)}-1)+1.$$
\end{theo}
\begin{proof}
  It is enough to show that for  any hyperplane $\Sigma$ of
  $\PG(W_k)$ not containing $\varepsilon_k^{gr}(\Delta_k)$ there
  are at least $s$ points in
  $\Phi=\varepsilon_k^{gr}(\Delta_k)\setminus\Sigma$ and then use
  \eqref{eq-codes}.
  Recall that when $q$ is odd, $\varepsilon_k^{gr}(\Delta_k)$ is
  not contained in any hyperplane.
  \par
  Let  $E$ be a point of $\Delta_k$ such that $\varepsilon^{gr}_k(E)\in\Phi$;
  as such, $E$ is a $k$--dimensional subspace contained in $Q(2n,q)$ and we
  can consider the star $St(E)\cong Q(2(n-k),q)$.
  Take $\Psi$ as a
  partial spread of maximum size of $St(E)$.
  For any $X,X'\in\Psi$, since $X$ and $X'$ are disjoint in
  $St(E)$, we have
  $X\cap X'=E$.  \par
  Furthermore, for any $X\in\Psi$,
  by Lemma~\ref{l:mnd},
  $\varepsilon_k^{gr}(\G_k(X))=
  \varepsilon_k^{gr}(\Delta_k)\cap W_k(X)$, where
  $W_k(X)=\langle \varepsilon_k^{gr}(\G_k(X))\rangle$.
  As $X$ is a $(n-1)$--dimensional projective space,
  we have also that $\varepsilon_k^{gr}(\G_k(X))$
   is isomorphic to the $k$--Grassmannian of an $n$--dimensional
   vector space.
   The hyperplane $\Sigma$ meets the subspace $W_k(X)$
   spanned by $\varepsilon_k^{gr}(\G_k(X))$
   in a hyperplane $\Sigma'$.
   By~\cite[Theorem 4.1]{N96}, wherein
   codes arising from projective Grassmannians are
   investigated and their minimal distance computed, we have
   $|W_k(X)\cap\Phi|\geq q^{k(n-k)}$.
   On the other hand, from
   \[ \varepsilon_k^{gr}(\G_k(X))\cap\varepsilon_k^{gr}(\G_k(Y))
   =\{\varepsilon_k^{gr}(E)\}, \]
   for any $X,Y\in\Psi$,
   it follows that $\varepsilon^{gr}_k(\Delta_k)$ has at least
   $\psi_{n-k}(q)(q^{k(n-k)}-1)+1$ points off $\Sigma$. This completes the proof.
\qed\end{proof}

Lemma~\ref{lemma-parametri} and Theorem~\ref{t:mnd}
together provide  Main~Result~\ref{main1}.
\par
In Subsection~\ref{s:tc} we determine the minimum
distance of $\cC^{gr}_{1,n}$ for $k=1$;
Subsections~\ref{s:dps2} and~\ref{s:dps3}
are dedicated to the case of dual polar spaces
of rank $2$ and $3$;
in these latter cases the minimum distance is precisely computed.

\subsection{Codes from polar spaces $\Delta_{1}$}
\label{s:tc}
If $k=1$, then $\Delta_k$ is just the orthogonal polar space and
$\varepsilon_1^{gr}$ is its natural embedding in $\PG(2n,q)$.
Hence, the code $\cC^{gr}_{1,n}$ is the code arising from the projective
system of the points of a non--singular parabolic quadric
$Q(2n,q)$ of $\PG(2n,q)$.
To compute its minimum distance, in light of \eqref{eq-codes},
it is enough to study the size of $Q(2n,q)\cap\Sigma$ where
$\Sigma$ is an arbitrary hyperplane of $\PG(2n,q)$.
This intersection achieves
its
maximum  at $(q^{2n-1}-1)/(q-1)+q^{n-1}$
when $Q(2n,q)\cap\Sigma$ is a non--singular hyperbolic quadric
$Q^+(2n-1,q)$, see~e.g.~\cite[Theorem 22.6.2]{HT91}.
Hence,
the parameters of the code $\cC^{gr}_{1,n}$ are
\[N=(q^{2n}-1)/(q-1);\,\, K=2n+1;\,\, d=q^{2n-1}-q^{n-1}.\]
The full weight enumerator can now be easily computed, using,
for instance,
\cite[Theorem 22.8.2]{HT91}.

\subsection{Dual polar spaces of rank $2$}
\label{s:dps2}
\subsubsection{Odd characteristic}
\label{s:2o}
Suppose that the characteristic of $\FF_q$ is odd.
By (\ref{emb1:pt2}) in Theorem~\ref{emb1}, the image
$\varepsilon^{gr}_2(\Delta_2)$ of the
dual polar space $\Delta_{2}$ under the Grassmann embedding
is isomorphic to the quadric
Veronese variety $\cV_2$ of $\PG(3,q)$, as embedded in
$\PG(9,q)$.
Length and dimension of the code $\cC^{gr}_{2,2}$ directly
follow from Theorem~\ref{emb1}. By Equation \eqref{eq-codes}, the minimum
distance of $\cC^{gr}_{2,2}$ is $|\cV_2|-m$, where
\[ m:=\max\{|\Sigma\cap\cV_2| : \Sigma \text{ is a hyperplane of $\PG(9,q)$}\}.
\]
It is well known, see~e.g.~\cite[Theorem 25.1.3]{HT91}, that there is
a bijection between the quadrics of $\PG(3,q)$ and the  hyperplane
sections of $\cV_2$; thus, in order to determine
$m$ we just need to consider the maximum cardinality of a quadric $Q$
in $\PG(3,q)$. This cardinality is $2q^2+q+1$, and
corresponds to the case in which $Q$ is the
union of two distinct planes.
Hence, we have the following theorem.
\begin{theo}\label{rk2-q-odd}
If $q$ is odd, then the code $\cC^{gr}_{2,2}$ is a
 $[N,K,d]_q$--linear code with the following parameters
\[N=(q^2+1)(q+1),\qquad K=10,\qquad d=q^2(q-1).\]
The full spectrum of its weights  is
$\{ q^3-q,q^3+q,q^3,q^3-q^2,q^3+q^2\}$.
\end{theo}
Theorem~\ref{rk2-q-odd} is part (\ref{mt2:i})
 of Main Result~\ref{main2} for $q$ odd.

\subsubsection{Even characteristic}
\label{s:2e}
Assume that $\FF_q$ has characteristic $2$.
By Theorem~\ref{emb2}, let $\N_2$ be the nucleus subspace of $\PG(W_2^{vs})$ such
that $\varepsilon_2^{gr}\cong\varepsilon_2^{vs}/\N_2$.
It is possible to choose a basis
$\fB$ of $V$ so that $\eta$
is given by $\eta(x_1,x_2,x_3,x_4,x_5)=x_1x_4+x_2x_5+x_3^2$;  by
\cite{CP2}, $\N_2$ can then be taken as the $1$--dimensional
subspace $\N_2=\langle(0,0,0,0,0,0,1,1,0,0)\rangle$.
Clearly, the code $\cC_{2,2}^{gr}$ has dimension
$K=\dim(\varepsilon_2^{gr})=\dim(\varepsilon^{vs}/\N_2)=9$.
To determine its minimum distance we use \eqref{eq-codes}; in
particular we need to compute
$|\varepsilon_2^{gr}(\Delta_2)\cap\Sigma|$
with $\Sigma$ an arbitrary hyperplane of the projective space
defined by $\langle\varepsilon_2^{gr}(\Delta_2)\rangle$.
Since $\langle\varepsilon_2^{gr}(\Delta_2)\rangle\cong
\langle\varepsilon_2^{vs}(\Delta_2)/\N_2\rangle$, we have
$\Sigma=\bar{\Sigma}/\N_2$ with $\bar{\Sigma}$ a hyperplane
of $\langle\varepsilon_2^{vs}(\Delta_2)\rangle=
\langle\cV_2\rangle$ containing $\N_2,$ where $\cV_2$,
as in Subsection~\ref{s:2o}, denotes the quadric Veronese variety
of $\PG(3,q)$ in $\PG(9,q)$.
As in the odd characteristic case, hyperplane sections of $\cV_2$
bijectively correspond to quadrics of $\PG(3,q)$ and the maximum cardinality for
a quadric $Q$ of a $3$--dimensional projective space is
attained when $Q$ is the union of two distinct planes, so it is
$2q^2+q+1$.
It is not hard to see that there actually exist degenerate quadrics
$Q$ of $\PG(3,q)$ which are
union of two distinct planes
and such that the corresponding hyperplane
$\Sigma_{Q}$ in $\langle\varepsilon_2^{vs}(\Delta_2)\rangle=
\langle\cV_2\rangle$ contains $\N_2$:
for instance, one can take the quadric $Q$
of equation $x_1x_2=0$.
Hence, $\Sigma_{Q}/\N_2$ is a hyperplane of
$\langle\varepsilon_2^{gr}(\Delta_2)\rangle\cong
\langle\varepsilon_2^{vs}(\Delta_2)/\N_2\rangle$.
As no line joining two points of $\varepsilon_2^{gr}(\Delta_2)$ passes
through $\N_2$,
\[ |\Sigma_{Q}\cap\cV_2|=|\Sigma_{Q}/\N_2\cap\varepsilon_2^{gr}(\Delta_2)|=|Q|=2q^2+q+1. \]
So, $|\varepsilon_2^{gr}(\Delta_2)\cap \Pi|\leq 2q^2+q+1$ for
every hyperplane $\Pi$ of $\langle\varepsilon_2^{gr}(\Delta_2)\rangle$.
This proves the following.
\begin{theo}
\label{t:nxt}
If $q$ is even,  then $\cC^{gr}_{2,2}$ is a linear $[N,K,d]_q$--code
with  parameters
\[N=(q^2+1)(q+1),\qquad K=9,\qquad d=q^2(q-1).\]
\end{theo}
Theorem~\ref{t:nxt} is part (\ref{mt2:i})
of Main Result~\ref{main2} for $q$ even.

\subsection{Dual polar spaces of rank $3$}
\label{s:dps3}
\subsubsection{Odd characteristic}
Here $\FF_q$ is assumed to have odd characteristic.
By (\ref{emb1:pt2}) in Theorem~\ref{emb1},
the image of the Grassmann embedding $\varepsilon_3^{gr}\cong
\varepsilon_3^{vs}$ spans a $34$--dimensional projective space.
Recall that the spin embedding $\varepsilon_3^{\spin}$  maps $\Delta_3$ into
the pointset $Q^+_7$ of a
non--singular hyperbolic quadric  of a $7$--dimensional projective
space; see~e.g.~\cite{Chevalley} and~\cite{BC}.
Hence, $\varepsilon_3^{{vs}}(\Delta_3)=
\nu_{2^3}(\varepsilon_3^{\spin}(\Delta_3))=\nu_{2^3}(Q^+_7)$
is a hyperplane section of $\langle \nu_{2^3}(\PG(7,q))\rangle$.
Using the correspondence induced by the quadratic Veronese embedding
$\nu_{2^3}:\PG(7,q)\to\PG(35,q)$  between
quadrics of
$\PG(7,q)$ and hyperplane sections of the quadric
Veronese variety $\cV_2$
we see that the pointset
$\varepsilon_3^{gr}(\Delta_3)\cong\nu_{2^3}(Q_7^+)$ is a hyperplane
section of $\cV_2$.

In order to determine the minimum distance $d$ of the code
$\cC^{gr}_{3,3}$ we need now to compute
\[m=\max \{|\Sigma\cap \varepsilon_3^{gr}(\Delta_3)| \colon \Sigma\, {\textrm{\,\,is\,\,a\,\,hyperplane\,\,of\,\,}}\PG(34,q)\}.\]
Note that
$|\Sigma\cap \varepsilon_3^{gr}(\Delta_3)|=|\Sigma\cap \nu_{2^3}(\ccQ^+_7)|$
and $\Sigma=\bar{\Sigma}\cap\langle\varepsilon_3^{gr}(\Delta_3)\rangle$,
where $\bar{\Sigma}$ is a hyperplane of $\langle\cV_2\rangle\cong\PG(35,q)$
different from
$\langle\nu_{2^3}(Q_7^+)\rangle=\langle\varepsilon_3^{gr}(\Delta_3)\rangle$.
Because of the Veronese
correspondence,
$\bar{\Sigma}=\nu_{2^3}(Q)$ for some quadric $Q$ of $\PG(7,q)$,
distinct from $Q_7^+$.
In particular,
\[ |\varepsilon_3^{gr}(\Delta_3)\cap\Sigma|=|Q_7^+\cap Q|. \]
Hence, in order to determine the minimum distance of the code, it suffices
to compute the maximum cardinality $m$ of $Q_7^+\cap Q$ with
$Q_7^+$   a given non--singular hyperbolic quadric
of $\PG(7,q)$
and $Q\neq Q_7^+$ any other quadric of $\PG(7,q)$.
\par
The study of the spectrum of the cardinalities of the intersection of
any two quadrics has been performed in~\cite{EHRS10},
in the context of functional codes of type $C_2(\ccQ^+)$, that is
codes defined by quadratic functions on quadrics;
see also~\cite[Remark 5.11]{L}.
In particular,
in~\cite{EHRS10}, the value of $m$
is determined by careful analysis of all possible intersection
patterns. Here we present an independent, different and
shorter, argument leading to the same conclusion, based on
elementary linear algebra. We point out
that our technique
could be extended to determine the full intersection spectrum of
two quadrics.
\begin{lemma}\label{intersection1}
Let $\ccQ^+$ be a given non--singular hyperbolic quadric of $\PG(2n+1,q)$.
If $\ccQ$ is any other quadric of $\PG(2n+1,q)$
not containing any generator of $\ccQ^+$,
then $|\ccQ\cap \ccQ^+|\leq (2q^{n}-q^{n-1}-1)(q^{n}+1)/(q-1)$.
\end{lemma}
\begin{proof}
The number of generators of $\ccQ^+$ is
$\kappa(n)=2(q+1)(q^2+1)\cdots (q^{n}+1)$.
By the assumptions,
any generator of $\ccQ^+$ meets
$\ccQ$ in a quadric $\ccQ'$ of $\PG(n,q)$.
It can be easily seen 
that $|\ccQ'|$ is maximal when $\ccQ'$ is the
union of two distinct hyperplanes; hence,
$|\ccQ'|\leq (2q^{n}-q^{n-1}-1)/(q-1)$.
Thus,
\[ |\ccQ^+\cap\ccQ|\leq \frac{(2q^{n}-q^{n-1}-1)}{(q-1)}\cdot
\frac{\kappa(n)}{\kappa(n-1)}
=
\frac{2q^{2n}-q^{2n-1}+q^{n}-q^{n-1}-1}{q-1}. \]
\qed\end{proof}

\begin{lemma}\label{lemma-intersection odd}
  Given a non--singular hyperbolic quadric $\ccQ^+$ in $\PG(2n+1,q)$,
  $q$ odd, we have
  \[ m=\max|\ccQ^+\cap \ccQ|=
  \frac{2q^{2n}-q^{2n-1}+2q^{n+1}-3q^{n}+q^{n-1}-1}{q-1}, \]
  as $\ccQ\neq\ccQ^+$ varies among all possible quadrics of
  $\PG(2n+1,q)$.
  This  number is attained only if the linear system generated
  by $\ccQ$ and $\ccQ^+$
  contains a quadric splitting in the union of two distinct hyperplanes.
\end{lemma}
\begin{proof}
Choose a reference system $\fB$
in $\PG(2n+1,q)$ wherein the quadric $\ccQ^+$ is represented by
the matrix
$C=\begin{pmatrix}
0& I\\
I& 0\\
\end{pmatrix}$, with
$I$ and
$0$ respectively the $(n+1)\times (n+1)$--identity and null matrices.
\par
If $\ccQ$ and $\ccQ^+$ were not to share any generator, then
the bound provided by Lemma~\ref{intersection1} on the size of their
intersection would hold.
Assume, instead, that  $\ccQ$ and $\ccQ^+$ have at least one
generator in common. We will determine the maximum intersection they
can achieve; as this will be larger then the
aforementioned bound, this will determine the actual maximum cardinality
that is attainable.
Under this hypothesis, we can
suppose that
$\ccQ$ is represented
with respect to $\fB$ by a matrix of the form
$S=\begin{pmatrix}
0& M\\
M^T& B\\
\end{pmatrix}$, with
$B$ a $(n+1)\times (n+1)$--symmetric matrix and $M$ an
arbitrary $(n+1)\times (n+1)$--matrix whose transpose is $M^T$.

Let $\begin{pmatrix}
X\\
Y\\
\end{pmatrix}$ be the coordinates
of a vector spanning a point of $\PG(2n+1,q)$  with $X$ and $Y$ column vectors of
length $n+1$.

Then, $\langle \begin{pmatrix}
X\\
Y\\
\end{pmatrix}\rangle \in \ccQ \cap \ccQ^+$  if and only if
\[ \begin{pmatrix}X^T& Y^T\end{pmatrix}\begin{pmatrix}
0& I\\
I& 0\\
\end{pmatrix}\begin{pmatrix}
X\\
Y\\
\end{pmatrix}=0\,\,\,\,{\textrm{and}}\,\,\,\, \begin{pmatrix}X^T&Y^T\end{pmatrix}
\begin{pmatrix}
0& M\\
M^T& B\\
\end{pmatrix}\begin{pmatrix}
X\\
Y\\
\end{pmatrix}=0, \]
which is equivalent to
\begin{equation}\label{intersection-eq}
\left\{\begin{array}{l}
X^TY=0\\
2X^TMY+Y^TBY=0.\\
\end{array}\right.
\end{equation}
We need to determine $M$ and $B$ as to maximise the number of
solutions of \eqref{intersection-eq}; in order to compute
this number, we consider \eqref{intersection-eq} as a family
of linear systems in the unknown $X$, with $Y$ regarded as a parameter.
Four cases
have to be investigated.
\begin{enumerate}
\item Take $Y=0$. Then, any $X$ is solution of \eqref{intersection-eq};
this accounts for $\frac{q^{n+1}-1}{q-1}$ points in the intersection.
\item
When $Y\neq0$ and $Y$ is not an eigenvector of $M$,
\eqref{intersection-eq} is a system of two independent equations in $n+1$
unknowns.
Hence, there are $q^{n-1}$ solutions for $X$.
If $N$ is the total number of eigenvectors of $M$,  the number of points in $\ccQ\cap \ccQ^+$
corresponding to this case is  $\frac{q^{n-1}(q^{n+1}-(N+1))}{q-1}$.
 \item
   If $Y$ is an eigenvector of $M$ and  $Y^TBY\neq0$, then
   \eqref{intersection-eq} has no solutions in $X$.
 \item
   Finally, suppose $Y$ to be an eigenvector of $M$ and $Y^TBY =0$. Then,
   there are $q^{n}$ values for $X$ fulfilling \eqref{intersection-eq}.
   Denote by $N_0$ the number
   of eigenvectors $Y$ of $M$ such that $Y^TBY =0$. Then, there are
   $\frac{q^{n}N_0}{q-1}$ distinct
   projective points in the intersection $\ccQ\cap \ccQ^+$ corresponding to this case.
\end{enumerate}
The preceding argument shows
\begin{equation}\label{intersection2}
\begin{array}{lll}
|\ccQ \cap \ccQ^+| &= &\frac{q^{n-1}(q^{n+1}-N-1)}{q-1} + \frac{q^{n}N_0}{q-1}+
\frac{q^{n+1}-1}{q-1} \\
 &  & \\
&= & \frac{ (qN_0-N)q^{n-1}}{q-1}+ \frac{(q^{n-1}+1)(q^{n+1}-1)}{q-1}.\\
\end{array}
\end{equation}

As $0\leq N_0\leq N\leq q^{n+1}-1$, the maximum of \eqref{intersection2} is
attained for the same values as
the maximum of $g(N_0,N):=(qN_0-N)/(q-1)$,
where $N_0$ and $N$ vary among all allowable values.
Clearly, when this quantity is maximal,
it has the same order of magnitude as $N_0$.
Several possibilities have
 to be considered:
 \begin{enumerate}[(i)]
   \item \label{o:c} $N_0=N=q^{n+1}-1$;
     then, the matrix $M$ has just one eigenspace of dimension $n+1$ and $B=0$.
From a geometric point of view this means $\ccQ^+\equiv\ccQ$.
   \item\label{ii:c} $N_0=2q^{n}-q^{n-1}-1$ and $N=q^{n+1}-1$;
           then,
           \[g_1:=g(2q^{n}-q^{n-1}-1,q^{n+1}-1)=q^{n}-1.\]
           The matrix $M$ has just one eigenspace $\cM_{n+1}$
           of dimension $n+1$ and $N_0/(q-1)$ is the maximum cardinality
           of a quadric of an
           $n$--dimensional projective space, corresponding to
           the union of two distinct hyperplanes.
    \item\label{i:c} $N_0=N=q^{n}+q-2$; then,
      \[g_2:=g(q^{n}+q-2,q^{n}+q-2)=q^n+q-2.\]
      The matrix $M$ has two distinct eigenspaces say $\cM_n$ and $\cM_1$,
      of dimension respectively $n$ and $1$ and eigenvalues
      $\lambda_n$ and $\lambda_1$.
 \end{enumerate}
All other possible values of $N_0$, corresponding to the cardinality of quadrics in a
 $(n+1)$--dimensional vector space, are smaller than $2q^n-q^{n-1}-1$.
 As $g_1\leq g_2$, the choice
 of (\ref{i:c}) gives the maximum cardinality.
\par
We now investigate the geometric configuration arising in Case (\ref{i:c}).
Let $\mathbb{U}=(u_1,u_2,\dots, u_{n+1})$ be a basis of eigenvectors for $M$
with $Mu_1=\lambda_1 u_1$ and
take $D$  as
a diagonalising matrix for $M$.
So, the column $D_i$ of $D$ is the eigenvector $u_i$
  for $i=1,\dots, n+1$ and $De_i=u_i$, with
  $\mathbb{E}=(e_1,e_2,\dots, e_{n+1})$  the canonical basis with respect to
  which  $M$ was originally written.
  We have $(M-\lambda_nI)De_1=(\lambda_1-\lambda_n)u_1$
  and $(M-\lambda_nI)De_i=0$ for $i=2,\dots, n+1$. Hence, $(M-\lambda_nI)D$
  is the null matrix except for the first column only. Thus,
  \[D^T(M-\lambda_nI)D=\begin{pmatrix}
    s_{0}  & 0 & \ldots & 0 \\
    s_1 & 0  & \ldots & 0 \\
    \vdots  &  & & \vdots \\
    s_{n} & 0   & \ldots & 0
    \end{pmatrix}.
    \]
    On the other hand, the matrix $B'=D^{T}BD$
    represents a quadric in $\PG(n,q)$ containing both the
    point $\langle(1,0,0,\cdots,0)\rangle$
    and the hyperplane of equation $x_1=0$, where the
    coordinates are written with respect to $\mathbb{U}$.
    Thus,
    \[
    D^{T}BD=\begin{pmatrix}
      0 & r_1 & \ldots & r_{n} \\
      r_1& 0  & \ldots&  0 \\
      \vdots & &   & \vdots \\
      r_{n}   & 0 & \ldots & 0 \\
  \end{pmatrix}. \]
  It is now straightforward to see that
  \[ \mathrm{rank}\left(\begin{pmatrix}
    D^{T} & 0 \\
    0 & D^{T}
    \end{pmatrix}\left(\begin{pmatrix}
      0 & M \\
      M^T & B
    \end{pmatrix}
    -\lambda_n\begin{pmatrix}
      0 & I \\
      I & 0
      \end{pmatrix}\right)\begin{pmatrix}
      D & 0 \\
      0 & D
      \end{pmatrix}\right)=2.
      \]
      In particular,
      as $\begin{pmatrix} D & 0 \\ 0 & D\end{pmatrix}$
      is invertible, also
      \[ \mathrm{rank}\left(\begin{pmatrix}
        0 & M \\
        M^T & B
      \end{pmatrix}
      -\lambda_n\begin{pmatrix}
        0 & I \\
        I & 0
      \end{pmatrix}\right)=2. \]
    Hence, the quadric $\ccQ'=\ccQ-\lambda_n\ccQ^+$ is
    union of two distinct hyperplanes.

    We remark  that also in the case of (\ref{ii:c}), the quadric
    $\ccQ-\lambda_{n+1}\ccQ^{+}$, where $\lambda_{n+1}$ is the
    eigenvalue of $M$ with multiplicity $n+1,$ is union of two hyperplanes,
    as $B$ has rank $2$.
\qed\end{proof}

\begin{theo}\label{n=k=3, odd}
For $q$ odd,  $\cC^{gr}_{3,3}$
is a $[N,K,d]_q$--linear code with the following parameters
\[N=(q^3+1)(q^2+1)(q+1),\qquad K=35,\qquad d=q^2(q-1)(q^3-1).\]
 \end{theo}
\begin{proof}
  By Lemma~\ref{lemma-intersection odd}, for $n=3$
  the maximum cardinality of the intersection
  of a hyperbolic quadric $\ccQ^+$ with any other quadric
  is
  $2q^5+q^4+3q^3+q+1$. The minimum distance follows
  from \eqref{eq-codes}.
\qed\end{proof}

Theorem~\ref{n=k=3, odd} is part (\ref{mt2:ii})
of Main Result~\ref{main2} for $q$ odd.

\subsubsection{Even characteristic}
We now consider the case $\FF_q=\FF_{2^r}$.
By (\ref{emb2:pt3}) in Theorem~\ref{emb2},
 $\varepsilon^{gr}_3(\Delta_3)\cong
(\varepsilon_3^{{vs}}/\cN_2)(\Delta_3)$,
where $\cN_2$ is the nucleus subspace of
$\langle \varepsilon_3^{{vs}}(\Delta_3)\rangle$.
Note that, by definition of quotient embedding, any line joining
two distinct points of $\varepsilon^{gr}_3(\Delta_3)$ is skew to
$\cN_2$.

As in the case of odd characteristic,
the spin embedding $\varepsilon_3^{\spin}$ maps
$\Delta_3$ to the pointset of a non--singular hyperbolic quadric $\ccQ^+_7$
of a $7$--dimensional projective space $\PG(7,q)$;
see~\cite{BC}. Hence,
by~\cite[Theorem 25.1.3]{HT91}, $\PG(W_3^{vs})=\langle
\varepsilon_3^{{vs}}(\Delta_3)\rangle=\langle
\nu_{2^3}(\varepsilon_3^{spin}(\Delta_3))\rangle =\langle
\nu_{2^3}(Q_7^+)\rangle$ is a hyperplane of the
$35$--dimensional projective space $\langle \nu_{2^3}(\PG(7,q))\rangle=
\langle\cV_2\rangle$, where $\cV_2$ is, as usual, the
quadric Veronese variety
of $\PG(7,q)$.

It is always possible to choose a reference system of $V$
wherein $\eta$ is given by
$\eta(x_1,x_2,x_3,x_4,x_5,x_6,x_7)=x_1x_5+x_2x_6+x_3x_7+ x_4^2$.
Let $(x_{i,j})_{1\leq i\leq j\leq 8}$ be the coordinates of a vector
$x$ in $\langle\cV_2\rangle $, written with respect to the basis
$(e_i\otimes e_j)_{1\leq i\leq j\leq 8}$ of $\langle\cV_2\rangle$,
with $(e_i)_i^8$ a basis of the vector space defining the $7$--dimensional
projective space $\langle \varepsilon_3^{\spin}(\Delta_3)\rangle$.
Then, by~\cite{CP2},  the equation of the
hyperplane $\langle \varepsilon_3^{{vs}}(\Delta_3)\rangle $
in $\langle\cV_2\rangle$ is
$x_{1,8}+x_{2,7}+x_{3,6}+x_{4,5}=0$, while $\cN_2$ can be represented
by the following system of $29$ equations:
\[ \left\{\begin{array}{lll}
x_{2,8}=x_{4,6},&x_{1,1}=0, &x_{1,5}=0\\
x_{2,3}=x_{1,4},&x_{2,2}=0,&x_{2,4}=0\\
x_{1,6}=x_{2,5},&x_{3,3}=0,&x_{2,6}=0\\
x_{1,7}=x_{3,5},&x_{4,4}=0,&x_{3,4}=0\\
x_{3,8}=x_{4,7},&x_{5,5}=0,&x_{3,7}=0\\
x_{5,8}=x_{6,7},&x_{6,6}=0,&x_{4,8}=0\\
x_{1,8}=x_{4,5},&x_{7,7}=0,&x_{5,6}=0\\
x_{2,7}=x_{4,5},&x_{8,8}=0,&x_{5,7}=0\\
x_{3,6}=x_{4,5},&x_{1,2}=0,&x_{6,8}=0 \\
   &x_{1,3}=0,& x_{7,8}=0. \\
\end{array}\right.
\]
Take now $\Sigma\neq W_3^{vs}$ as an arbitrary hyperplane of $\langle
\nu_{2^3}(\PG(7,q))\rangle$ containing $\cN_2$. Then,
$\Sigma$ has equation of
the form
 \[\sum_{1\leq i\leq j\leq 8}a_{i,j}x_{i,j}=0,\]
 with the coefficients $a_{i,j}$ fulfilling
\[
\left\{\begin{array}{ll}
a_{1,4}=a_{2,3}, &
a_{1,6}=a_{2,5}\\
a_{1,7}=a_{3,5}, &
a_{2,8}=a_{4,6}\\
a_{3,8}=a_{4,7}, &
a_{5,8}=a_{6,7}\\
\multicolumn{2}{l}{a_{1,8}+a_{2,7}+a_{3,6}+a_{4,5}=0.}\\
\end{array}
\right. \]
By~\cite[Theorem 25.1.3]{HT91},
there is a quadric $\ccQ_{\Sigma}$ of the $7$--dimensional projective
space $\langle \varepsilon_3^{\spin}(\Delta_3)\rangle=\langle \ccQ^+_7\rangle$
such that
$\Sigma\cap\cV_2=\nu_{2^3}(\ccQ_{\Sigma})$.
Since $\cN_2\subset \Sigma$ and $\cN_2$ is skew with respect to
$\varepsilon_3^{gr}(\Delta_3)$,
\[ |\ccQ_{\Sigma}\cap\ccQ^+_7|= |\Sigma/\cN_2\cap\varepsilon_3^{gr}(\Delta_3)|. \]
Observe that $\langle \varepsilon_3^{gr}(\Delta_3)\rangle\cong
\PG(W^{vs}_3/\cN_2)$ is a $27$--dimensional projective space
and $\Sigma/\cN_2$ is an
arbitrary hyperplane of $\langle \varepsilon_3^{gr}(\Delta_3)\rangle$.
With the notation just introduced, we prove the following.
\begin{lemma}
As
$\ccQ_{\Sigma}$ varies among all the quadrics of $\PG(7,q)$ corresponding
to hyperplanes $\Sigma$ of $\langle\nu_{2^3}(\PG(7,q))\rangle$
containing $\cN_2$,
\[ m=\max|\ccQ_{\Sigma} \cap\varepsilon_3^{\spin}(\Delta_3)|=2q^5+q^4+2q^3+q^2+q+1.
\]
\end{lemma}
\begin{proof}
Suppose the pointset of $\varepsilon_3^{\spin}(\Delta_3)$ to
be that of the hyperbolic quadric $\ccQ^+_7$ of  equation
$x_1x_8+x_2x_7+x_3x_6+x_4x_5=0$.
As the bound of Lemma~\ref{intersection1} holds also in
even characteristic, we
need to consider
those hyperplanes $\Sigma=\nu_{2^3}(\ccQ_{\Sigma})$ of $\langle \cV_2\rangle$
containing $\cN_2$ and
corresponding to quadrics $\ccQ_{\Sigma}$ of the
$7$--dimensional projective space
$\langle \varepsilon_3^{\spin}(\Delta_3)\rangle$
with at least one generator in
common with
$\ccQ^+_7$. In particular, we can assume $\Sigma$ to have
equation \[\sum_{1\leq i\leq j\leq 8} a_{i,j}x_{i,j}=0,\]
where the coefficients $a_{i,j}$ satisfy
\[\left\{
\begin{array}{lll}
a_{1,4}=a_{2,3}, &
a_{1,6}=a_{2,5}, &
a_{1,7}=a_{3,5}, \\
a_{2,8}=a_{4,6}, &
a_{3,8}=a_{4,7}, \\
\multicolumn{3}{l}{a_{1,8}+a_{2,7}+a_{3,6}+a_{4,5}=0,}\\
{a_{i,j}=0} & \multicolumn{2}{l}{\text{when $5\leq i\leq j$}.}
\end{array}
\right. \]
Hence, the quadric $\ccQ_{\Sigma}$ has equation
$\sum_{1\leq i\leq j\leq 8} a_{i,j}x_ix_j=0$,
with the coefficients $a_{i,j}$ fulfilling the previous conditions.
Thus, $\cN_2$ is contained in the hyperplane $\Sigma=\nu_{2^3}(\ccQ_{\Sigma})$,
while $\ccQ_{\Sigma}\cap \ccQ^+_7$ contains
the $3$--dimensional projective space of equations $x_1=x_2=x_3=x_4=0$.

Rewrite the equation of $\ccQ_{\Sigma}$ in a more compact form as
\[Y^TM^TX+ \sum_{1\leq i\leq j\leq 4} a_{i,j}x_ix_j=0,\]
where
\[X=\begin{pmatrix}
x_1\\
x_2\\
x_3\\
x_4\\
\end{pmatrix},\,\, Y=\begin{pmatrix}
x_5\\
x_6\\
x_7\\
x_8\\
\end{pmatrix},
M=\begin{pmatrix}
a_{1,5}& a_{1,6}&a_{1,7}&a_{1,8}\\
a_{1,6}& a_{2,6}&a_{2,7}&a_{2,8}\\
a_{1,7}& a_{3,6}&a_{3,7}&a_{3,8}\\
a_{4,5}& a_{2,8}&a_{3,8}&a_{4,8}\\
\end{pmatrix} \]
with $a_{1,8}+a_{2,7}+a_{3,6}+a_{4,5}=0$ and $a_{1,4}=a_{2,3}$.

We can also write the equation
of $\ccQ^+_7$ as $Y^TJX=0$ with $J=\begin{pmatrix}
0& 0&0&1\\
0& 0&1&0\\
0& 1&0&0\\
1& 0&0&0\\
\end{pmatrix}$.

Arguing as in the proof of Lemma~\ref{lemma-intersection odd},
let $\langle\begin{pmatrix}
X\\
Y\\
\end{pmatrix}\rangle$ be a point
of $\PG(7,q)$.
Then, $\langle\begin{pmatrix}
X\\
Y\\
\end{pmatrix}\rangle\in \ccQ_{\Sigma}\cap\ccQ^+_7$ if, and only if,
\begin{equation}\label{intersection-eq-even}
\left\{\begin{array}{l}
Y^TJX=0\\
Y^TM^TX+ \sum_{1\leq i\leq j\leq 4} a_{i,j}x_ix_j=0,\\
\end{array}\right.
\end{equation}
where $J$, $M$ are as previously defined.
\par
Since $J^2=I$, if we put $\bar{M}:=JM^T$ and $\bar{Y}^T:=Y^T J$,
System \eqref{intersection-eq-even} becomes as follows,
where we have also included the conditions on the coefficients $a_{i,j}$:
\begin{equation}\label{intersection-eq-evenbis}
\left\{\begin{array}{l}
\bar{Y}^TX=0\\
\bar{Y}^T \bar{M}X+\sum_{1\leq i\leq j\leq 4} a_{i,j}x_ix_j=0\\
\trace(\bar{M})=0, \,\,a_{2,3}=a_{1,4}.\\
\end{array}\right.
\end{equation}
System \eqref{intersection-eq-evenbis} is the analogue
of System \eqref{intersection-eq} in
Lemma~\ref{lemma-intersection odd}  for $n=3$,
with the further restrictions $\trace(\bar{M})=0$ and $a_{2,3}=a_{1,4}$.
Hence, it is possible to perform the same analysis
as before,
in order to determine the number of its solutions. The
maximum is achieved
when $\bar{M}$ admits a unique eigenspace of dimension
$4$, as in Cases  (\ref{o:c}) and (\ref{ii:c})
of Lemma~\ref{lemma-intersection odd}.
This means that $\bar{M}$ is similar to a diagonal matrix
$\diag(\lambda_1,\lambda_1,\lambda_1,\lambda_1)$, hence
$\trace(\bar{M})=\trace(\diag(\lambda_1,\lambda_1,\lambda_1,\lambda_1))=
4\lambda_1=0$ and the trace condition is satisfied.

Furthermore, if the coefficients $a_{i,j}$
in $\sum_{1\leq i\leq j\leq 4} a_{i,j}x_ix_j=0$ are all
$0$, then $\ccQ_{\Sigma}=\ccQ^+_7$;
this is the analogue of Case (\ref{o:c}) of Lemma~\ref{lemma-intersection odd}.
Note that  Case (\ref{i:c})
of Lemma~\ref{lemma-intersection odd}
cannot happen, as if $\bar{M}$ were to admit two
eigenspaces of dimensions respectively $1$ and $3$,
then it would  be similar to a
diagonal matrix
$\diag(\lambda_1,\lambda_2,\lambda_2,\lambda_2)$, $\lambda_1\not= \lambda_2$.
However,
$\trace(\bar{M})=\trace(\diag(\lambda_1,\lambda_2,\lambda_2,\lambda_2))=
\lambda_1+3\lambda_2=0$ gives $\lambda_1=\lambda_2$ --- a contradiction.

When the vectors satisfying the equation
$\sum_{1\leq i\leq j\leq 4} a_{i,j}x_ix_j=0$ represent
points lying on two distinct planes of a $3$--dimensional projective space,
we have the analogue of Case (\ref{ii:c})
of Lemma~\ref{lemma-intersection odd} and this achieves the maximum
intersection size.
\par
We have thus shown that
the maximum value $m$ for $|\ccQ_{\Sigma}\cap\ccQ^+|$ is attained for
$\bar{M}$ similar to the diagonal matrix
$\diag(\lambda_1,\lambda_1,\lambda_1,\lambda_1)$
and $m=2q^5+q^4+2q^3+q^2+q+1$.
\qed\end{proof}

\begin{theo}\label{n=k=3, even}
For $q$ even,
the code $\cC^{gr}_{3,3}$ is a $[N,K,d]_q$--linear
code with 
\[N=(q^3+1)(q^2+1)(q+1),\qquad K=28,\qquad d=q^5(q-1).\]
\end{theo}

Theorem~\ref{n=k=3, even} is part (\ref{mt2:ii})
of Main Result~\ref{main2} for $q$ even.

\section{Projective and polar caps}
\label{sec:2}
In this section $\F$ is an arbitrary, possibly infinite, field.
A \emph{projective cap} of $\PG(n,\F)$ is a set $\cC$ of points of $\PG(n,\F)$
which is met by no line of $\PG(n,\F)$ in more than $2$ points.
A  generalisation to an arbitrary
point--line geometry $\Gamma=(\cP,\cL)$ is as follows: an
\emph{$(m,v)$--set}
$\cC\subseteq\cP$ is a set of $m$ points which is met by any $\ell\in\cL$ in
\emph{at most} $v$ points.

Clearly, when $\Gamma$ is a linear space, non--trivial $(m,v)$--sets can exist
only for $v\geq 2$; however,
when not all of the points of $\Gamma$ are collinear,
$(m,1)$--sets are also interesting (consider, for instance,
the case of ovoids of  polar spaces).
\par
In the present section we shall be dealing exclusively
with $(m,2)$--sets, henceforth called in brief $m$--caps.
When $\Gamma=\PG(r,\F),\G_k$ or $\Delta_k$ we speak respectively of
projective, Grassmann or polar $m$--caps.

In Theorem~\ref{thm-cap}, it will be shown
that  the whole pointset of
a dual polar space $\Delta_n$ is mapped by the Grassmann
embedding into  a projective cap,
even if, clearly, the
full pointset $\Delta_k$ for any $k\leq n$ could never be
a polar cap of itself, as $\Delta_k$, for $n>1,$ contains lines.
For $k<n$, the Grassmann embedding $\varepsilon_k^{gr}$ is a
projective embedding, that is it maps lines of $\Delta_k$ onto projective lines;
thus, $\varepsilon_k^{gr}(\Delta_k)$ cannot be a cap.
However, in
Theorem~\ref{grassmann projective} and
Corollary~\ref{polar projective} we shall show
that Grassmann and polar caps are mapped by $\varepsilon_k^{gr}$ into
projective caps; see also~\cite{EMS} and~\cite{BCE} for
caps contained in classical varieties.
This is significant as,
when a geometry $\Gamma$ is projectively
embedded in a larger geometry, say $\Gamma'$
and not all the points of $\Gamma$ are
collinear, then there might be $m$--caps of $\Gamma$ which
are not inherited by $\Gamma'$.

\begin{theo}
\label{t21}
Let $1\leq k\leq n$.
If  $\cC$ is a polar $m$--cap of $\Delta_k$, then $\cC$
is a Grassmann $m$--cap of $\G_k$.
\end{theo}
\begin{proof}
Let $P_1,P_2$ and $P_3$ be three distinct points of $\cC$.
By way of
contradiction, suppose $P_1$, $P_2$ and $P_3$ to be collinear in
$\G_k$. So, $P_1,P_2$ and $P_3$ are three $k$--dimensional
totally singular subspaces of $V$ with $\dim(P_1\cap P_2\cap
P_3)=k-1$ and $\dim\langle P_1,P_2,P_3\rangle=k+1$. Put $S:=\langle
P_1,P_2,P_3\rangle$.

If $\F=\F_2$, then $S=P_1\cup P_2\cup P_3$ is a singular subspace;
hence, $P_1$, $P_2$
and $P_3$ are collinear in $\Delta_k$. This contradicts the hypothesis.

If $\F\not=\F_2$, take $x\in S\setminus (P_1\cup P_2 \cup P_3)$ and
$y\in P_1\setminus (P_1\cap P_2\cap P_3)$.
The line $\langle x,y\rangle $ meets $P_2$ and $P_3$ in distinct points,
say
$y_2\in P_2\setminus P_3$ and $y_3\in P_3\setminus P_2$,
as each $P_i,\,1\leq i\leq 3$ is a hyperplane in $S$ and $x\not\in
(P_1\cup P_2 \cup P_3)$. Then, the line $\langle x,y\rangle$ has three
distinct singular points. Necessarily, $\langle x,y\rangle$ is a
singular line; thus, $x$ is a singular point and $S$ is a totally
singular subspace.

For $1\leq k<n$, this means that $P_1$, $P_2$ and $P_3$ are
collinear in $\Delta_k$, contradicting the hypothesis on $\cC$.

For $k=n$ we would have determined a totally singular subspace $S\leq V$
of dimension $n+1$.
This is, again, impossible, as the maximal singular subspaces
of $V$ have dimension $n$.
\qed\end{proof}

\begin{theo}
\label{grassmann projective}
Let $1\leq k\leq n$.
If $\cC$ is a Grassmann $m$--cap of $\G_k$, then $e_k^{gr}(\cC)$
 is a projective cap of $\PG(W_k)$.
\end{theo}
\begin{proof}
Let $P_1,P_2$ and $P_3$ be three distinct points of $\cC$. Put
$\bar{P}_1:=e_k^{gr}(P_1)$, $\bar{P}_2:=e_k^{gr}(P_2)$ and
$\bar{P}_3:=e_k^{gr}(P_3)$.  By way of contradiction, suppose
$\bar{P}_1$, $\bar{P}_2$ and $\bar{P}_3$ to be collinear in
$\PG(W_k)$.
The image $e_k^{gr}(\G_k)$ of the Pl\"ucker embedding
$e_k^{gr}$ of $\G_k$ is the intersection of (possibly degenerate)
quadrics of $\PG(W_k)$. Since, by assumption, the projective line
$\langle \bar{P}_1, \bar{P}_2 \rangle$ meets $e_k^{gr}(\G_k)$ in
three distinct points $\bar{P}_1$, $\bar{P}_2$ and $\bar{P}_3$, we
have $\langle \bar{P}_1, \bar{P}_2 \rangle\subseteq
e_k^{gr}(\G_k)$, that is  $\bar{P}_1$, $\bar{P}_2$ and $\bar{P}_3$ are
on a line of $e_k^{gr}(\G_k)$. By~\cite[Theorem 24.2.5]{HT91},
$P_1$, $P_2$ and $P_3$ should be on a line of
$\G_k$ and, thus, collinear in $\G_k$ --- a contradiction.
\qed\end{proof}

\begin{co}\label{polar projective}
Let $1\leq k\leq n$.
If  $\cC$ is a polar $m$--cap of $\Delta_k$,
then $\varepsilon_k^{gr}(\cC)$ is a projective $m$--cap of $\PG(W_k^{gr})$.
\end{co}

\begin{theo} \label{thm-cap} The image $\varepsilon_n^{gr}(\Delta_n)$
of the dual polar space $\Delta_n$ under the Grassmann embedding
$\varepsilon_n^{gr}$ is a projective cap of $\PG(W_n^{gr})$.
\end{theo}
\begin{proof}
We prove that $\varepsilon_n^{gr}(\Delta_n)$ does not contain any three
collinear points.  By way of contradiction, suppose
$\varepsilon_n^{gr}(P_1)$, $\varepsilon_n^{gr}(P_2)$ and
$\varepsilon_n^{gr}(P_3)$ to be three collinear points in
$\PG(W_n^{gr})$ and put $\ell:=\langle \varepsilon_n^{gr}(P_1),
\varepsilon_n^{gr}(P_2)\rangle$. The image $e_n^{gr}(\G_n)$ of the
projective Grassmannian $\G_n$ by the Pl\"ucker embedding $e_n^{gr}$ is
a variety obtained as the intersection of (possibly degenerate)
quadrics of $\PG(W_n)$.
Since $\ell$ is a projective line containing three points of
$e_n^{gr}(\G_n)$, then $\ell\subset e_n^{gr}(\G_n)$.
By~\cite[Theorem 24.2.5]{HT91}, its pre--image $r=(e_n^{gr})^{-1}(\ell)$
is a line
of $\G_n$. Hence, $P_1,$ $P_2$ and $P_3$ are three distinct
points of $\Delta_n$ lying on the line $r$ of
$\G_n$. This means that there are three distinct  maximal
subspaces $p_1,p_2$ and $p_3$ of $V,$ totally singular with
respect to $\eta$, intersecting in a
$(n-1)$--dimensional subspace and spanning a $(n+1)$--dimensional
subspace of $V$. This configuration is, clearly, impossible.
\qed\end{proof}

Main Result~\ref{main3} is a consequence of Corollary~\ref{polar projective}
and Theorem~\ref{thm-cap}.
\par
As recalled in Section~\ref{Sect-1.1},
when $\F=\F_q$, the pointset of the dual polar
space $\Delta_n$ is the set of all $(q^n+1)(q^{n-1}+1)\cdots (q+1)$
$n$--dimensional subspaces of $V$ totally singular with respect to
$\eta$. Thus, we get the following corollary.
\begin{co}\label{cor-cap}
Suppose $n\geq 2$ and $\F=\F_q$ a finite field. Then,
\begin{enumerate}[(i)]
\item\label{c1p1}
  For $q=p^h$, $p>2$, the pointset $\varepsilon_n^{gr}(\Delta_n)$ is
  a cap of $\PG({{2n+1}\choose n}-1,q)$ of size $(q^n+1)(q^{n-1}+1)\cdots
  (q+1)$.
\item\label{c1p2}
  For $q=2^h$, the pointset $\varepsilon_n^{gr}(\Delta_n)$ is a cap
  of $\PG({{2n+1}\choose n}-{{2n+1}\choose {n-2}}-1,q)$ of size
  $(q^n+1)(q^{n-1}+1)\cdots (q+1)$.
\end{enumerate}
\end{co}
\begin{proof}
By Theorem~\ref{thm-cap},  $\varepsilon_n^{gr}(\Delta_n)$
is a cap of $\PG(W_n^{gr})$.
Part (\ref{c1p1}) of the corollary  follows from Part (\ref{emb1:pt1}) of
Theorem~\ref{emb1}.
Part (\ref{c1p2}) follows from Part (\ref{emb2:pt1})
of Theorem~\ref{emb2}.
\qed\end{proof}
We remark that Part (\ref{c1p1}) of Corollary~\ref{cor-cap} can
also be proved using Part (\ref{emb1:pt2}) of Theorem~\ref{emb1}
together with the well--known result of~\cite{TVM2003} showing that the
quadric Veronesean of $\PG(n,q)$ is a cap of $\PG(n(n+3)/2,q).$  

\section{Construction of a polar cap of $\Delta_k$}
\label{construction}
In this section $\FF$ can be any, possibly infinite,
 field of odd characteristic.
We shall determine a family of $k$--dimensional
subspaces of $V$ totally singular with respect to $\eta$
providing a polar cap of $\Delta_k$, for $k\leq n$.
Observe that the caps we construct in this section
all actually fulfil the stronger condition $v=1$, that is
no $2$ of their points are on a line of $\Delta_k$;
furthermore, as all of the results of Section~\ref{sec:2} for $v\leq2$
apply, they determine caps
of the ambient projective space  by Theorem~\ref{polar projective}.

Up to a multiplicative non--zero constant, it is possible
to choose without loss of
generality a basis
${\mathbb B}=(e_1,e_2,\ldots,e_{2n+1})$
for $V$ in which the quadratic form $\eta$ is
given by
\[\eta(x_1,\ldots , x_{2n+1}) = \sum_{i=1}^nx_ix_{n+i} + x_{2n+1}^2.\]
Denote by $f_{\eta}$ the symmetric bilinear form obtained by polarising $\eta$
and by $\perp$ the  associated  orthogonality relation.
Given  $I:=\{1,\ldots,2n+1\}$, write ${I\choose k}$ for the set of
all $k$--subsets of $I$.

For any set of indices
$J=\{j_1,j_2,\ldots,j_k\}\subset I,\, j_1<j_2<\ldots< j_k,$
define
$e_J=e_{j_1}\wedge e_{j_2}\wedge\cdots \wedge e_{j_k}$.
The set
$B_{\wedge}:=(e_J)_{J\in {I\choose k}}$ is, clearly,
a basis of $W_k$.
For any $i$,
 with $1\leq i\leq 2n+1$, let
\[ i':=\begin{cases}
  i+n & \text{ if $1\leq i\leq n$}\\
  i-n & \text{ if $n< i\leq 2n$ } \\
  2n+1 & \text{ if $i=2n+1$. }
\end{cases} \]
Observe that, with $f_{\eta}$ defined
as above and $1\leq i\leq 2n$, we always get $f_{\eta}(e_i,e_{i'})=1$.
Thus, the pair $\{e_i,e_{i'}\}$ is a \emph{hyperbolic pair of vectors};
see~\cite[Chapter 3]{Ar}.
We shall say
that $\{i,i'\}$ is a \emph{hyperbolic pair of indices}
if the corresponding set $\{e_i,e_{i'}\}$ is
a hyperbolic pair of vectors.

\begin{lemma}\label{hyperbolic pairs}
Let $k\leq n$ and $r\leq \lfloor\frac{k}{2}\rfloor$. Suppose $J$ to be a
$k$--subset of $I$ containing $r$ hyperbolic pairs of
indices.
The following statements hold:
\begin{enumerate}[(1)]
\item\label{e:c1} If $2n+1\not\in J$, then there exists
$\{m_1,m_2,\ldots,m_r\}\subseteq \{1,2,\ldots,n\}$ such that $e_{m_i}\in
\{e_{j}\}_{j\in J}^{\perp}$ for every $1\leq i\leq r$.
\item\label{e:c2} If $2n+1\in J$, then there exists
$\{m_1,m_2,\ldots,m_r,\ell\}\subseteq \{1,2,\ldots,n\}$ such that $e_{t}
\in \{e_{j}\}_{j\in J}^{\perp}$ for every $t\in\{m_1,\ldots,m_r\}\cup\{\ell\}$.
\end{enumerate}
\end{lemma}
\begin{proof}
Write $J\cap \{1,2,\ldots,n\}=\{j_1,\ldots,j_r,
j_{r+1},j_{r+2},\ldots,j_{r+s}\}$ and write
$J\cap\{1',2',\ldots,n'\}=\{j_1',j_2',\ldots,j_r',
j_{r+s+1}', j_{r+s+2}',\ldots,j_{k-r}'\}$.
Let
 $U=\{1,2,\ldots,n\}\setminus (J\cup J')$,
 where $J'=\{j'\colon j\in J\}$.
\begin{enumerate}[(1)]
\item
If $2n+1\not\in J$, then
$|U|=n-(r+k-2r)=n-k+r\geq r$, since $n-k\geq 0$.
Hence, there exists a subset $M_r=\{m_1,m_2,\ldots,m_r\}$ of $U$
of cardinality $r$.
Clearly, $f_{\eta}(e_j,e_{m_i})=0$ for every $m_i\in M_r$ and every $j\in J$.
\item
  Since
 $2n+1\in J$, we have $|U|=n-(r+k-(2r+1))=n-k+r+1\geq r+1$, as $n-k\geq 0$.
 Hence, there exists a subset $\baM_r=\{m_1,m_2,\ldots,m_r,\ell\}$
 of $U$ of cardinality $r+1$. Clearly,
 $f_{\eta}(e_j,e_t)=0$ for every $t\in\baM_r$
 and every $j\in J$. \qed
\end{enumerate}
\end{proof}

\subsection{First construction: $2n+1\not\in J$}
\label{2n+1 not in J}

Suppose
$J=\{j_1,j_2,\ldots, j_r, j_1',j_2',\ldots, j_r'\}\cup \bar{J}\subset I$,
where $\bar{J}$ does not contain any hyperbolic pair of indices, $|J|=k$
and $2n+1\not\in J$.
By (\ref{e:c1}) in Lemma~\ref{hyperbolic pairs}, there exists
$M_r=\{m_1,m_2,\ldots,m_r\}\subseteq\{1,2,\ldots,n\}$
such that  $e_{m_i}\in \{e_{j}\}_{j\in J}^{\perp}$.
We will construct a family of $2^r$ totally singular $k$--dimensional
subspaces of $V$ from these $m_i\in M_r$ as follows. Fix any bijection
$\tau:\{j_1,j_2,\ldots,j_r\}\to M_r$  and put
\begin{equation}\label{construction-1}
\begin{array}{lll}
X_{\emptyset,\tau}&:=&\langle e_{j_1}+e_{\tau(j_1)},e_{j_2}+e_{\tau(j_2)},\ldots,
e_{j_r}+e_{\tau(m_r)},\\
 & &e_{j_1'}-e_{\tau(j_1)'}, e_{j_2'}-e_{\tau(j_2)'},\ldots,
 e_{j_r'}-e_{\tau(j_r)'}, \{e_j\}_{j\in \bar{J}}\rangle.\\
\end{array}
\end{equation}
For every non--empty subset $S$ of $M_r$
define $X_{S,\tau}$ to be the $k$--dimensional subspace
of $V$ spanned by the same vectors
as $X_{\emptyset,\tau}$ in \eqref{construction-1}
except that when $\tau(j_i)\in S$, the vectors $e_{j_i}+e_{\tau(j_i)}$ and
$e_{j_i'}-e_{\tau(j_i)'}$
are respectively replaced by $e_{j_i}-e_{\tau(j_i)'}$ and $e_{j_i'}+e_{\tau(j_i)}$.
For simplicity in the following arguments, as well as in Subsection
\ref{2n+1 not in J}, we shall always assume $m_i=\tau(j_i)$ and
write just $X_{S}$ for $X_{S,\tau}$.
For an example and an explicit description,
 see Table~\ref{Table1}.
\begin{table}[h]
\[\begin{array}{lll}
X_{\emptyset}&:=&\langle e_{j_1}+e_{m_1},e_{j_2}+e_{m_2},\ldots, e_{j_r}+e_{m_r},\\
 & &e_{j_1'}-e_{m_1'}, e_{j_2'}-e_{m_2'},\ldots, e_{j_r'}-e_{m_r'}, \{e_j\}_{j\in \bar{J}}\rangle;\\[4pt]
X_{m_1}&:=&\langle e_{j_1}-e_{m_1'},e_{j_2}+e_{m_2},\ldots,e_{j_r}+e_{m_r},\\
 & &e_{j_1'}+e_{m_1}, e_{j_2'}-e_{m_2'},\ldots, e_{j_r'}-e_{m_r'}, \{e_j\}_{j\in \bar{J}}\rangle;\\[4pt]
X_{m_2}&:=&\langle e_{j_1}+e_{m_1},e_{j_2}-e_{m_2'},\ldots, e_{j_r}+e_{m_r}\\
 & &e_{j_1'}-e_{m_1}, e_{j_2'}+e_{m_2},\ldots, e_{j_r'}-e_{m_r'}, \{e_j\}_{j\in \bar{J}}\rangle;\\[4pt]
&\cdots  & \\[4pt]
X_{m_1,\dots,m_r}&:=&\langle e_{j_1}-e_{m_1'},e_{j_2}-e_{m_2'},\ldots,e_{j_r}-e_{m_r'},\\
 & &e_{j_1'}+e_{m_1}, e_{j_2'}+e_{m_2},\ldots,e_{j_r'}+e_{m_r}, \{e_j\}_{j\in \bar{J}}\rangle.\\
\end{array}\]
\caption{Subspaces for $2n+1\not\in J$}
\label{Table1}
\end{table}
\begin{theo} \label{cap1}
The set $\fX_k:=\{X_S\}_{S\subseteq M_r}$ is a polar $2^r$--cap of $\Delta_k$.
\end{theo}
\begin{proof}
Clearly $|\fX_k|=2^r$.
We now prove $\fX_k\subset\Delta_{k}$ and that
no two distinct elements of $\fX_k$
are collinear in $\Delta_k$.
By Lemma ~\ref{hyperbolic pairs}, it is straightforward to see that
for any $S\subseteq M_r$,
the subspace $X_S$ is totally singular with respect to $\eta$.
Let $S$ and $T$ be two
arbitrary distinct subsets of $M_r$. Since $S\neq T$, there
exists $u\in\{1,2,\ldots,r\}$ such that $m_u\in S$ and $m_u\not\in
T$. So,   $\langle e_{j_u}- e_{m_u'}, e_{j_u'}+
e_{m_u}\rangle \not\subseteq X_S\cap X_T$. It follows that
the distance $d(X_S,X_T):=k-\dim(X_S\cap X_T)$
between $X_S$ and $X_T$, regarded as
points
of the
collinearity graph of $\cG_k$, is at least $2$.
As the collinearity graph of $\Delta_k$ is a subgraph of that of
$\cG_k$, this yields the result.
\qed\end{proof}
We observe that by Theorem \ref{cap1}, $\fX_k$ is also a
$(2^r,1)$--set of $\cG_k$.

For each $S\subseteq M_r$, denote by $B_S$ the set formed by the
first $2r$ generators of $X_S$, ordered as in
Table~\ref{Table1}, and by  $\overline{X}_S=\langle B_S\rangle$
the subspace of $X_S$ spanned by $B_S$.
\begin{co}\label{cap2}
The set $\overline{\fX}_{2r}=\{\overline{X}_S\}_{S\subseteq M_r}$
 is a polar $2^r$--cap of $\Delta_{2r}$.
\end{co}

Given an arbitrary $S\subseteq M_r$, the elements of $B_S$
can be described as follows:
\[ \begin{array}{c}
 e_{j_1}+ (-1)^{{\chi_S}(m_1)}e_{m_1+ n{\chi_S}(m_1)},\ldots,e_{j_r}+(-1)^{\chi_S(m_r)}e_{m_r+ n{\chi_S}(m_r)}, \\
 e_{j_1'}+ (-1)^{\chi_S(m_1)+1}e_{m_1 +n(1-{\chi_S}(m_1))}, \ldots,e_{j_r'}+(-1)^{\chi_S(m_r)+1}e_{m_r + n(1-{\chi_S}(m_r))},
\end{array} \]
where $\chi_S$ is the characteristic function of $S$,
that is $\chi_S(x)=1$ if $x\in S$ and $\chi_S(x)=0$ if $x\not\in S$,
and, as before,
$x':=x+n$.

The Grassmann embedding $\varepsilon_{2r}^{gr}$ applied to any of the
singular subspaces $\overline{X}_S=\langle B_S\rangle$
determines a point
$\varepsilon_{2r}^{gr}(\overline{X}_S)=\langle \wedge^{2r}B_S\rangle$
of $\PG(W_{2r}^{gr})$, with
\begin{small}
$$\label{eq1}
\begin{array}{l}
\wedge^{2r}B_S=
(e_{j_1}+ (-1)^{{\chi_S}(m_1)}e_{m_1+ n{\chi_S}(m_1)})\wedge \cdots \wedge (e_{j_r}+(-1)^{\chi_S(m_r)}e_{m_r+ n{\chi_S}(m_r)})\wedge\\
\wedge (e_{j_1'}+ (-1)^{\chi_S(m_1)+1}e_{m_1 +n(1-{\chi_S}(m_1))})\wedge \cdots \wedge (e_{j_r'}+(-1)^{\chi_S(m_r)+1}e_{m_r + n(1-{\chi_S}(m_r))}). \\
\end{array}
$$
\end{small}
Hence,  $\wedge^{2r}B_S$ is a sum of vectors of the form
$\sigma_S(K)\cdot e_K$, where $\sigma_{S}(K) = \pm 1$ and $K \subseteq
\{j_\ell, j'_\ell, m_\ell, m'_\ell\}_{\ell=1}^r$ has size $2r$ and contains at most
$r$ hyperbolic pairs of indices given by either $\{j_\ell,j'_\ell\}$
or $\{m_\ell,m'_\ell\}$.

It is possible to write $\wedge^{2r}B_S$ in a more convenient way by
expanding the wedge products.
To this end, let $T =
\{t_1,\ldots , t_r\}\subseteq\{j_1,\ldots ,j_r,m_1,\ldots ,m_r\}$ with
$t_\ell \in \{j_\ell,m_\ell\}$ for $1\leq\ell\leq r$,
$T' = \{t'_1,\ldots , t'_r\}$
and denote by ${\cT}_r$ the family of all such sets $T$. The
mapping sending every $T\in{\cT}_r$ to $T\cap M_r$ is a
bijection between ${\cT}_r$ and the family of all the subsets of
$M_r$. Hence, $|{\cT}_r|=2^r$.

Consider $U=\{u_1,\ldots,u_{2r}\}\subseteq
\{j_1,\ldots ,j_r,m_1,\ldots ,m_r, j_1',\ldots ,j_r',m_1',\ldots ,m_r'\}$
such that
$|\{\{i,i'\}\subset U\}|<r$ and denote by $\cU$ the family of all
such sets.
\par
In other words, every set $T\cup T'$ with $T\in\cT_r$ is
made up of precisely $r$ hyperbolic pairs of indices, while any
 $U\in\cU$ is made up of \emph{at most} $r-1$ hyperbolic
pairs of indices.
 Then,
\begin{equation}\label{eq wedge1}
\wedge^{2r}B_S=\sum_{T\in \cT_r} \sigma_S(T)e_{T, T'}+
\sum_{U\in \cU} \sigma_S(U)e_U,
\end{equation} where $e_{T, T'}:=e_T\wedge e_{T'}$ and $\sigma_S(T)$,
$\sigma_S(U)$ are
shorthand notations
for $\sigma_S(T\cup T')$ and $\sigma_S(U\cup U')$, respectively.

Put
\begin{equation}\label{xsi-prima costruzione}
\xi_S:= \sum_{T\in \mathcal{T}_r} \sigma_S(T)e_{T,T'}.
\end{equation}
 In particular, $\xi_{\emptyset}:= \sum_{T\in \cT_r}
\sigma_{\emptyset}(T)e_{T,T'},$ where, as it can be easily seen,
$\sigma_{\emptyset}(T)=(-1)^{|T\cap M_r|}$.

By Corollaries~\ref{polar projective} and~\ref{cap2},
$\varepsilon_{2r}^{gr}(\overline{\fX}_{2r})=\{\varepsilon_{2r}^{gr}(\overline{X}_S)\}_{S\subseteq M_r}$ is a
projective cap of $\PG(W_{2r}^{gr})$.
The function sending
$\varepsilon_{2r}^{gr}(\overline{X}_S)$ to $\xi_S$
is a bijection between $\varepsilon_{2r}^{gr}(\overline{\fX}_{2r})$ and
the set
$\{\xi_S\}_{S\subseteq M_r}$.

\subsection{Second construction: $2n+1 \in J$}
\label{2n+1 in J}
We now move to Case (\ref{e:c2}) of  Lemma~\ref{hyperbolic
pairs}. In close analogy to Section~\ref{2n+1 not in
J}, we will introduce a family of $2^r$ totally singular $k$--dimensional
subspaces of $V$. Most of the results previously proved
hold unchanged  when $2n+1\in J$.

Let $J=\{j_1,j_2,\ldots,j_r, j_1',j_2',\ldots,j_r',2n+1\}\cup
\bar{J}\subset I$, where $\bar{J}$ does not contain any hyperbolic pair of
indices
and $|J|=k$. By (\ref{e:c2}) in Lemma~\ref{hyperbolic pairs}, there exists
\[ \baM_r=\{m_1,m_2,\ldots,m_r,\ell\}\subseteq\{1,2,\ldots,n\} \]
such that
$e_{t}\in \{e_{j}\}_{j\in J}^{\perp}$ for any $t\in\baM_r$.

Put
\[\begin{array}{lll}
\mathcal{X}_{\emptyset}&:=&\langle e_{j_1}+e_{m_1},e_{j_2}+e_{m_2},\ldots,e_{j_r}+e_{m_r},\\
 & &e_{j_1'}-e_{m_1'}, e_{j_2'}-e_{m_2'},\ldots,e_{j_r'}-e_{m_r'}, e_\ell+e_{2n+1}-e_{\ell'}, \{e_j\}_{j\in \bar{J}}\rangle.\\
\end{array}\]
Clearly, $\mathcal{X}_{\emptyset}$ is totally singular.

As before, for every non--empty subset $S$
of $M_r=\{m_1,m_2,\ldots,m_r\}$ define
$\mathcal{X}_S$ to be the $k$--dimensional subspace of $V$ spanned by the same
vectors as $\mathcal{X}_{\emptyset}$, except that if $m_i\in S$,
then $e_{j_i}+e_{m_i}$ and $e_{j_i'}-e_{m_i'}$
are respectively replaced by
$e_{j_i}-e_{m_i'}$ and $e_{j_i'}+e_{m_i}$. For more details, see
Table~\ref{Table2}.
\begin{table}[h]
\[ \begin{array}{lll}
\mathcal{X}_{\emptyset}&:=&\langle e_{j_1}+e_{m_1},e_{j_2}+e_{m_2},\ldots,e_{j_r}+e_{m_r},\\
 & &e_{j_1'}-e_{m_1'}, e_{j_2'}-e_{m_2'},\ldots,e_{j_r'}-e_{m_r'}, e_\ell+e_{2n+1}-e_{\ell'},\{e_j\}_{j\in \bar{J}}\rangle;\\[4pt]
\mathcal{X}_{m_1}&:=&\langle e_{j_1}-e_{m_1'},e_{j_2}+e_{m_2},\ldots,e_{j_r}+e_{m_r},\\
 & &e_{j_1'}+e_{m_1}, e_{j_2'}-e_{m_2'},\ldots,e_{j_r'}-e_{m_r'}, e_\ell+e_{2n+1}-e_{\ell'},\{e_j\}_{j\in \bar{J}}\rangle;\\[4pt]
\mathcal{X}_{m_2}&:=&\langle e_{j_1}+e_{m_1},e_{j_2}-e_{m_2'},\ldots,e_{j_r}+e_{m_r}\\
 & &e_{j_1'}-e_{m_1}, e_{j_2'}+e_{m_2},\ldots,e_{j_r'}-e_{m_r'},
 e_\ell+e_{2n+1}-e_{\ell'},\{e_j\}_{j\in \bar{J}}\rangle;\\[4pt]
&\cdots & \\[4pt]
\mathcal{X}_{m_1,\ldots,m_r}&:=&\langle e_{j_1}-e_{m_1'},e_{j_2}-e_{m_2'},\ldots,e_{j_r}-e_{m_r'},\\
 & &e_{j_1'}+e_{m_1}, e_{j_2'}+e_{m_2},\ldots,e_{j_r'}+e_{m_r},
 e_\ell+e_{2n+1}-e_{\ell'}, \{e_j\}_{j\in \bar{J}}\rangle.\\
\end{array} \]
\caption{Subspaces for $2n+1\in J$}
\label{Table2}
\end{table}

We thus determine $2^r$ totally singular $k$--dimensional subspaces of
$V$ each being at distance at least $2$ from any other, when
regarded as points
in the collinearity graph of $\Delta_k$.
Hence, the following analogue of Theorem~\ref{cap1} holds.
 \begin{theo}\label{cap3}
   The set $\fX_{k}'=\{\mathcal{X}_S\}_{S\subseteq M_r}$ is a polar $2^r$--cap
   of $\Delta_k$.
\end{theo}

Arguing as in Subsection~\ref{2n+1 not in J},
let $\widehat{B}_S$ be the set consisting of the first $2r+1$ generators
of $\mathcal{X}_S$
and $\overline{\mathcal{X}}_S=\langle \widehat{B}_S\rangle$ be the subspace of
$\mathcal{X}_S$ spanned by $\widehat{B}_S$. In other words,
$\widehat{B}_S:=B_S\cup \{e_\ell+e_{2n+1}-e_{\ell'}\},$ with $B_S$ defined
as in Subsection~\ref{2n+1 not in J}.

The following corresponds to Corollary~\ref{cap2}.
\begin{co}\label{ccap2}
  The set $\overline{\fX}_{2r+1}'=\{\overline{\mathcal{X}}_S\}_{S\subseteq M_r}$
  is a polar
  $2^r$--cap of $\Delta_{2r+1}$.
\end{co}

For any $S\subseteq M_r$,
apply the Grassmann embedding $\varepsilon_{2r+1}^{gr}$ to the
singular subspaces $\overline{\mathcal{X}}_S=\langle
\widehat{B}_S\rangle$. Hence,
$\varepsilon_{2r+1}^{gr}(\overline{\mathcal{X}}_S)=\langle
\wedge^{2r+1}\widehat{B}_S\rangle$ is the point of $\PG(W_{2r+1}^{gr})$
spanned by the vector
$\wedge^{2r+1}\widehat{B}_S:=\wedge^{2r}B_S\wedge
(e_\ell+e_{2n+1}-e_{\ell'})$.

Expanding $\wedge^{2r+1}\widehat{B}_S$, we get an analogue of
\eqref{eq wedge1}:
\begin{equation}\label{eq wedge2}
\wedge^{2r+1}\widehat{B}_S=\sum_{T\in \mathcal{T}_r} \sigma_S(T)e_{T, T',2n+1}+
\sum_{\bar{U}\in \bar{\mathcal{U}}} \sigma_S(\bar{U})e_{\bar{U}},
\end{equation}
where $e_{T, T',2n+1}:=e_T\wedge e_{T'}\wedge e_{2n+1}$,
$\bar{U}\subseteq \bar{\mathcal{U}}= \mathcal{U}\cup \{l,l'\},$
$|\{(i,i')\subset \bar{U}\}|<r$, $ |\bar{U}|=2r+1$; the sets $T, T', {\cal
T}_r$ and $\mathcal{U}$ are defined as in Section~\ref{2n+1 not in J}. The
coefficients $\sigma_S(T)$ and $\sigma_S(\bar{U})$ are
$\pm 1$. Put
\begin{equation}\label{xsi-seconda costruzione}
\overline{\xi}_S:=\sum_{T\in \mathcal{T}_r} \sigma_S(T)e_{T, T',2n+1}.
\end{equation}

By Corollaries~\ref{polar projective} and~\ref{ccap2},
$\varepsilon_{2r+1}^{gr}(\overline{\fX}'_{2r+1})=\{\varepsilon_{2r+1}^{gr}(\overline{\mathcal{X}}_S)\}_{S\subseteq M_r}$
is a projective cap
of $\PG(W_{2r+1}^{gr})$.
The function sending any element
 $\varepsilon_{2r+1}^{gr}(\overline{\mathcal{X}}_S)$
 to the vector $\overline{\xi}_S$
is a bijection between $\varepsilon_{2r+1}^{gr}(\overline{\fX}'_{2r+1})$ and
$ \{\overline{\xi}_S\}_{S\subseteq M_r}$.

Observe that Main  Result~\ref{main4} is contained in
Corollaries~\ref{cap2} and~\ref{ccap2}.

\section{Hadamard matrices and codes from caps}
\label{Hadamard-Sylvester}
Recall that a
Hadamard matrix of order $m$ is an $(m\times m)$--matrix $H$ with
entries $\pm 1$ such that $HH^t=mI$, where $I$ is the $(m\times m)$--identity
matrix. Hadamard matrices have been
widely investigated, as their existence, for $m>2$, is equivalent to
that of extendable symmetric $2$--designs with parameters
$(m-1,\frac{1}{2}m-1,\frac{1}{4}m-1)$; see~\cite{CvL}, and also
\cite[Theorem 4.5]{HP}.
It is well known that
the point--hyperplane design of $PG(n,2)$ is a Hadamard
$2-(2^{n+1}-1,2^n-1,2^{n-1}-1)$ design; any of the corresponding Hadamard
matrices is called a \emph{Sylvester matrix}; see~\cite[Example
1.31]{CvL}.
Indeed,
the so called recursive \emph{Kronecker product
  construction}, see~\cite[Theorem 3.23]{HP}, as
\[
 S_1=\begin{pmatrix}
   1 & 1 \\
   1 &-1 \\
   \end{pmatrix}, \qquad
 S_n=S_{n-1}\otimes\begin{pmatrix}
  1 & 1 \\
  1 & -1 \\
  \end{pmatrix}
  \]
  always gives a Sylvester matrix.
In this section we shall show how it is possible
to associate a Hadamard matrix of order $2^r$
to any polar cap $\overline{\fX}_{2r}$ of
 $\Delta_{2r}$ and $\overline{\fX}_{2r+1}'$ of
 $\Delta_{2r+1}$.
Recall that $\overline{\fX}_{2r}$ and $\overline{\fX}_{2r+1}'$ are
introduced respectively in
Corollaries~\ref{cap2} and~\ref{ccap2} of Section~\ref{construction}.
In particular, we shall make use  of the
 vectors $\xi_S$ and $\bar{\xi}_S$ therein computed respectively in equations
\eqref{xsi-prima costruzione} and \eqref{xsi-seconda costruzione}.
We shall also introduce an order relation on the points of the cap
itself, in order to prove that this matrix can be obtained by the recursive
Sylvester construction.
This will also provide a direct connection with first order Reed--Muller
codes; for more details, see also~\cite{AK0}.


At first, we need to take into account the two cases of
Sections~\ref{2n+1 not in J} and~\ref{2n+1 in J} separately.
We adopt the same notation as is those sections.

For $2n+1\not\in J$,
put ${\cB}_S := \{\sigma_S(T)e_{T,T'}\}_{T\in{\cT}_r}$ for $S \subseteq M_r$;
see Equation \eqref{xsi-prima costruzione}.  Then,
${\cB}_S$ is a basis of the linear space $L_{{\cT}_r} := \langle
e_{T,T'}\rangle_{T\in{{\cT}_r}}$. In particular, $\cB_\emptyset$
is a basis of $L_{{\cT}_r}$
and $\xi_{S}=\sum_{T\in \mathcal{T}_r}\sigma_S(T)e_{T,T'}\in L_{{\cT}_r}$.
Thus, we can consider the coordinates
$\{(\xi_{S})_T\}_{T\in{{\cT}_r}}$ of $\xi_{S}$ with respect to
${\cB}_\emptyset$. Clearly, $(\xi_{S})_T
= \sigma_S(T)\sigma_{\emptyset}(T)$.
Observe that, while we have selected ${\cB}_\emptyset$ as a basis,
the result holds for any arbitrary fixed basis of the form $\cB_S$.

If $2n+1\in J$, let $\overline{\cB}_S=\{\sigma_S(T)e_{T,T',2n+1}\}_{T\in\cT_r}$;
see Equation \eqref{xsi-seconda costruzione}.
Then, $\overline{\cB}_S$ is a basis of the linear space
$\bar{L}_{{\cT}_r}=\langle e_{T,T',2n+1}\rangle_{T\in{\cT_r}}$.
In particular, $\overline{\cB}_{\emptyset}$ is a basis of
$\bar{L}_{{\cT}_r}$ and $\bar{\xi}_S\in\bar{L}_{{\cT}_r}$;
thus
we consider the coordinates
$\{(\bar{\xi}_S)_T\}_{T\in{\cT_r}}$ of $\bar{\xi}_S $
with respect to the basis ${\overline{\cB}}_\emptyset$
of $\bar{L}_{{\cT}_r}$. Again, we have $(\bar{\xi}_S)_T =
\sigma_S(T)\sigma_{\emptyset}(T)$.

Let $A_{\emptyset,r}$ be the $(2^r\times 2^r)$--matrix defined as
follows. The rows are indexed by the subsets of $M_r=\{m_1,\ldots,m_r\}$
and the columns by the members of ${\cT}_r$. For $S\subseteq M_r$
and $T\in{{\cT}_r}$ the $T$--entry of the row
$R_S$ corresponding to $S$ is equal to $(\xi_{S})_T=
\sigma_S(T)\sigma_{\emptyset}(T)$ when $2n+1\not\in J$
and  $(\overline{\xi}_S)_T=\sigma_S(T)\sigma_{\emptyset}(T)$
when $2n+1\in J$.
In particular, every entry of
$A_{\emptyset,r}$ is either $1$ or $-1$ and all entries in the row
$R_\emptyset$ are equal to $1$.
\begin{lemma}\label{matrix}
\
\begin{enumerate}
\item
When $2n+1\not\in J$,
$A_{\emptyset,r}={((\xi_{S})_T)}_{\begin{subarray}{l}
S\subseteq M_r\\
T\in\cT_r\\
\end{subarray}} \text{ with } (\xi_{S})_T=(-1)^{|S\cap T|}$.
\item
When $2n+1\in J$,
$A_{\emptyset,r}={((\bar{\xi}_{S})_T)}_{\begin{subarray}{l}
S\subseteq M_r\\
T\in\cT_r\\
\end{subarray}} \text{ with }  (\bar{\xi}_{S})_T=(-1)^{|S\cap T|}$.
\end{enumerate}
\end{lemma}
\begin{proof}
Suppose $2n+1\not\in J$. The proof for the case $2n+1\in J$ is
entirely analogous.

Take  $R=\{m_1,\ldots,m_{\ell-1}\}\subseteq M_r$ and
let $S=R\cup\{m_{\ell}\}\subseteq M_r$.
Observe that $\xi_{S}$ is obtained from $\xi_{R}$ by
replacing
$e_{j_\ell}+e_{m_\ell}$ and $e_{j_\ell'}-e_{m_\ell'}$
by respectively $e_{j_\ell}-e_{m_\ell'}$ and $e_{j_\ell'}+e_{m_\ell}$
in $\wedge^{2r}B_{R}$.
Clearly, if $m_{\ell}\not\in T$, we have
$(\xi_S)_T=(\xi_R)_T$, as
$\xi_{R}$ and $\xi_{S}$ have exactly the same
components with respect to all the vectors $e_{T,T'}$ 
which do not contain the term $e_{m_\ell}$.
On the other hand,
when $m_\ell\in T$, the
sign of the component of $e_{T,T'}$ must be swapped; thus,
$(\xi_{S})_T=-(\xi_{R})_T$.
As $\xi_{S}$ can be obtained from the sequence
\[ \xi_\emptyset \to \xi_{m_1} \to \xi_{m_1,m_2} \to \cdots \to \xi_{R}\to \xi_S \]
and $\xi_{\emptyset}=\mathbf{1}$, we have
$(\xi_{S})_T=(-1)^{|S\cap T|}$.
This proves the lemma.
\qed\end{proof}

\begin{theo}
\label{t:hadamard}
 The matrix $A_{\emptyset,r}$ is Hadamard.
\end{theo}
\begin{proof}
By Lemma~\ref{matrix}, $(\xi_{S})_T=(-1)^{\mathbf{S}\cdot\mathbf{T}}$, where
$\mathbf{S}$ and $\mathbf{T}$ are the incidence vectors of
$S$ and $T\cap M_r$ with respect to $M_r$ and $\cdot$ denotes the usual
inner product.
The result now is a consequence of~\cite[Lemma 4.7, page 337]{St}.
\qed\end{proof}

In particular, $(\xi_{S})_T=1$ if, and only if, $S$ and $T$ share an
even number of elements.
\begin{co}
\label{c:php}
  The design associated to the matrix $A_{\emptyset,r}$ is the point--hyperplane
  design of $\PG(r,2)$; in particular, $A_{\emptyset,r}$ is a Sylvester
  matrix.
\end{co}
\begin{proof}
It is well known that for any hyperplane $\pi$ of $\PG(r,2)$,
there is a point
$P_{\pi}\in\PG(r,2)$ such that
\[ \pi=\{ X\in\PG(r,2): P_{\pi}\cdot X=0 \}, \]
with $\cdot$ the usual inner product of $\PG(r,2)$ and
$P_{\pi}=(p_1,p_2,\ldots,p_{r+1})$, $X=(x_1,x_2,\ldots,x_{r+1})$
binary vectors.
In particular, $X\in\pi$ if and only if
$P_{\pi}\cdot X=|\{i: p_i=x_i\}|\pmod{2}=0$,
that is to say if and only if the vectors $P_{\pi}$ and $X$ have an even
number of $1$'s in common. By Lemma~\ref{matrix},
it is now straightforward to see that the matrix
$A_{\emptyset,r}'$ obtained from $A_{\emptyset,r}$ by deleting the all--$1$
row and column
and replacing $-1$ with  $0$ contains the incidence vectors of the symmetric
design of points and hyperplanes of a projective space $\PG(r,2)$.
\qed\end{proof}

Recall that equivalent Hadamard matrices give isomorphic Hadamard
designs; the converse, however, is not true in general.

As anticipated,
we now show how the rows and columns of $A_{\emptyset,r}$ or, equivalently, the
points of the polar caps $\overline{X}_S$, might be ordered
as to be able to describe it in terms of the
Kronecker product construction.

Since both the rows and the columns of $A_{\emptyset,r}$ can be indexed
by the subsets of $M_r$,
(for the columns we just consider
 $T\cap M_r$ with $T\in \cT_r$) it is enough
to introduce a suitable order $<_r$
on  the set $2^{M_r}$ of all subsets of $M_r$.
We proceed  in a recursive way
as follows:
\begin{itemize}
  \item for $r=1$, define
    $\emptyset<_1\{m_1\}$;
  \item suppose we have ordered $2^{M_{r-1}}$, then
    for any $X,Y\subseteq M_r$, we say $X<_rY$ if and only if
    \begin{enumerate}
    \item $X<_{r-1}Y$ when $m_r\not\in X\cup Y$;
    \item $m_r\not\in X$ and $m_r\in Y$;
    \item $m_r\in X\cap Y$ and $(X\setminus\{m_r\})<_{r-1}(Y\setminus\{m_r\})$.
    \end{enumerate}
\end{itemize}
Observe that  $<_r$, when restricted to $M_{r-1}$, is the same as
$<_{r-1}$. Thus,
we shall drop the subscript from $<_{r}$,
given that no ambiguity may arise.
\par

As examples, for $r=2$ we have
\begin{small}
\[ \emptyset<\{m_1\}<\{m_2\}<\{m_1,m_2\}, \]
\end{small}
while, for $r=3$,
\begin{small}
\[ \emptyset<\{m_1\}<\{m_2\}<\{m_1,m_2\}<\{m_3\}<\{m_1,m_3\}<\{m_2,m_3\}<\{m_1,m_2,m_3\}. \]
\end{small}
The minimum under $<$ is always $\emptyset$, and the maximum $M_r$.
Using the order induced by $<$ on both the rows and the columns of
$A_{\emptyset,r}$ we prove the following.
\begin{theo}
\label{matrice-A_r}
For any $r>1$ we have
$A_{\emptyset,r}=A_{\emptyset,r-1}\otimes A_{\emptyset,1}$.
\end{theo}
\begin{proof}
The matrix $A_{\emptyset,r}$ encodes the parity of the intersection of
subsets of $M_r$; as we took the same order for columns and rows,
$A_{\emptyset,r}$ is clearly symmetric.
We now show that
\[ A_{\emptyset,r}=\begin{pmatrix}
  A_{\emptyset,r-1} & A_{\emptyset,r-1} \\
  A_{\emptyset,r-1} & -A_{\emptyset,r-1}
  \end{pmatrix}=A_{\emptyset,r-1}\otimes\begin{pmatrix} 1 & 1 \\ 1 & -1
  \end{pmatrix}. \]
Indeed, the elements indexing the first $2^{r-1}$ rows and columns
of $A_{\emptyset,r}$ are all subsets of $M_{r-1}$ in the order
given by $<_{r-1}$.
Thus, the minor they determine is indeed $A_{\emptyset,r-1}$.
Observe now that if $m_r\in Y$ and $m_r\not\in X$, then
\[ (-1)^{|X\cap Y|}=(-1)^{|X\cap (Y\setminus\{m_r\})|}. \]
In particular, the entry in row $1\leq x\leq 2^{r-1}$ and column
$2^{r-1}<y\leq 2^r$  is the same as that in row $x$ and
column $y-2^{r-1}$. It follows that the minor of $A_{\emptyset,r}$
comprising the first $2^{r-1}$ rows and the last $2^{r-1}$ columns
is also $A_{\emptyset,r-1}$. By symmetry, this applies also to
the minor consisting of the last $2^{r-1}$ rows and the first $2^{r-1}$
columns.
Finally, consider an entry in row $2^{r-1}<x\leq 2^{r}$ and column
$2^{r-1}<y\leq 2^r$. By  definition of $<_r$,
the sets $X,Y$ indexing this entry are
$X=X'\cup\{m_r\}$ and $Y=Y'\cup\{m_r\}$ where $X'$ and $Y'$
index the entry in row $x-2^{r-1}$ and column $y-2^{r-1}$.
In particular, as $|X\cap Y|=|X'\cap Y'|+1$,
\[ (-1)^{|X\cap Y|}=-(-1)^{|X'\cap Y'|}. \]
It follows that this minor of $A_{\emptyset,r}$
is $-A_{\emptyset,r-1}$.

By Lemma~\ref{matrix},
\[ A_{\emptyset,1}=\begin{pmatrix}
1&1\\
1&-1
\end{pmatrix}. \]
The theorem now follows by recursion.
\qed\end{proof}

By Theorem~\ref{matrice-A_r}, the matrix $A_{\emptyset,r}$ is obtained
by the Sylvester construction.  As a corollary of
Theorem~\ref{matrice-A_r}, the codes associated to the  caps
constructed in Section~\ref{construction} are Reed--Muller
codes of the first order.

\medskip

\noindent
{\bf Acknowledgements.} The authors wish to express their gratitude to
Antonio Pasini for his very helpful remarks on a first version of this paper.
In particular, the idea to exploit partial spreads to get an improvement
for the bound of the minimal distance contained in Main Theorem \ref{main1}
is due to him.

\bigskip

\noindent

\end{document}